\documentclass[11pt,a4paper,reqno]{amsart}

\newcommand{\arrowdraw}[2]{\draw [postaction={decorate}] (#1)--(#2);}
\newcommand{\arrowdrawb}[2]{\draw [postaction={decorate}] (#2)--(#1);}

\newcommand{\dedgedraw}[2]{\draw [densely dotted] (#1)--(#2);}

\usepackage{tikz}
\usetikzlibrary{decorations.markings}
\usetikzlibrary{arrows,matrix}
\usepgflibrary{arrows}

\usepackage[all]{xy}
\input xy
\xyoption{all}
\usepackage{xy}

\usepackage{clrscode}
\usepackage{mathrsfs,amssymb}

\usepackage{amsmath, amsthm}

\newcommand{\topp}[1]{\left\lceil{#1}\right\rceil}
\newcommand{\bott}[1]{\left\lfloor{#1}\right\rfloor}

\newtheorem{theorem}{Theorem}[section]
\newtheorem{lemma}[theorem]{Lemma}
\newtheorem{corollary}[theorem]{Corollary}

\newtheorem{proposition}[theorem]{Proposition}
\newtheorem{definition}[theorem]{Definition}
\newtheorem{example}[theorem]{Example}

\newcommand{\rinf}{\mathbb{R}^\infty}
\newcommand{\outball}{\overrightarrow{\mathcal{B}}}
\newcommand{\inball}{\overleftarrow{\mathcal{B}}}

\newcommand{\gr}{\mathcal{R}}
\newcommand{\gl}{\mathcal{L}}
\newcommand{\gh}{\mathcal{H}}
\newcommand{\gd}{\mathcal{D}}
\newcommand{\gj}{\mathcal{J}}

\begin{document}

\title[Hyperbolicity for directed graphs and monoids]{A strong geometric hyperbolicity property for directed graphs and monoids}

\maketitle

\begin{center}

    ROBERT GRAY\footnote{
School of Mathematics, University of East Anglia, Norwich NR4 7TJ, UK.    
Email \texttt{Robert.D.Gray@uea.ac.uk}. Research supported by
FCT and FEDER, project POCTI-ISFL-1-143 of Centro de \'{A}lgebra da
Universidade de Lisboa, and by the project PTDC/MAT/69514/2006. He
gratefully acknowledges the support of EPSRC grant EP/F014945/1 and the
hospitality of the University of Manchester during a visit to Manchester.} \ and \ MARK KAMBITES\footnote{School of Mathematics, University of Manchester, Manchester M13 9PL, England. Email \texttt{Mark.Kambites@manchester.ac.uk}.
Research supported by an RCUK Academic Fellowship and
by EPSRC grant EP/F014945/1.} \\
\end{center}

\begin{abstract}
We introduce and study a strong ``thin triangle'' condition for directed graphs,
which
generalises the usual notion of hyperbolicity for a metric space. We prove
that finitely generated left cancellative monoids whose right Cayley graphs
satisfy this condition must be finitely presented with polynomial Dehn
functions, and hence word problems in $\mathcal{NP}$. Under the additional
assumption of right cancellativity (or in some cases the weaker condition of
bounded indegree), they also admit algorithms for
more fundamentally semigroup-theoretic decision problems such as Green's
relations $\mathcal{L}$, $\mathcal{R}$, $\gj$, $\gd$ and the corresponding
pre-orders.

In contrast, we exhibit a right cancellative (but not left cancellative)
finitely generated monoid (in fact, an infinite class of them) whose
Cayley graph is a essentially a tree (hence hyperbolic in our sense and
probably any reasonable sense), but which is not even recursively
presentable. This
seems to be strong evidence that no geometric notion of hyperbolicity will be strong
enough to yield much information about finitely generated monoids in absolute generality.
\end{abstract}

\section{Introduction}\label{sec_intro}

Over the past half century, combinatorial group theory has been increasing
dominated by ideas from geometry. One of the most successful aspects is the
theory of \textit{word hyperbolic groups}, in which a simple, combinatorial
notion of negative curvature for a group Cayley graph is used to give tight
control on the geometric, combinatorial and even computational structure of
the group \cite{Gromov87}.
The question naturally arises of whether these geometric techniques are
particular to groups, or if they apply to a wider class of monoids.
There are numerous equivalent characterisations of word hyperbolic groups,
which lead to different (non-equivalent) ways in which one could define a word
hyperbolic monoid.
 For example, a finitely presented group is word hyperbolic exactly if it has linear
\textit{Dehn function}, and this property can also be studied for monoids
 \cite{Cain12,Otto98}. A beautiful theorem of Gilman \cite{Gilman02} characterises hyperbolic groups
as those which admit \textit{context-free multiplication tables}: several
authors have studied the class of monoids satisfying this and similar
conditions
\cite{Cain12b,Cain12c,Duncan04,Hoffmann10}. Another approach is to
treat a monoid Cayley graph as an undirected graph, and require that it
be a hyperbolic \cite{Cain12b,Cassaigne09,KambitesHypCS} metric space. In general, the directional
information in a monoid or semigroup Cayley graph is of crucial importance in understanding the algebraic
structure (and most especially the \textit{ideal} structure) of the monoid,
and it is unreasonable to expect much control from any condition on
the Cayley graph which disregards the directed structure. For example, it is
easily seen that any finitely generated semigroup with a zero element will
have undirected Cayley graph of bounded diameter $2$. However, this
approach does seem to have some merit in classes of semigroups with a very
restricted ideal structure, such as \textit{completely simple semigroups}
\cite{KambitesHypCS}.

Here we propose a new way in which hyperbolicity conditions can be
extended to cancellative monoids, in a way which is geometric but does
not artificially impose a metric structure on the monoid. Specifically, we
introduce
(in Section~\ref{sec_semimetric} below) and
study a strong ``directed thin triangle'' condition for \textit{directed} graphs,
which
generalises the usual notion of hyperbolicity for a metric space.
We prove (in Section~\ref{sec_hypmonoids}) that a finitely generated left
cancellative monoid whose Cayley graph satisfies this condition
must be finitely presented with polynomial Dehn function, and hence
solvable (indeed, non-deterministic polynomial time) word problem. Under
additional right cancellativity assumptions, we also
show (Section~\ref{sec_greens}) that such a monoid admits algorithms for more
fundamentally semigroup-theoretic decision problems, including Greens'
equivalence relations $\mathcal{L}$, $\mathcal{R}$, $\gj$, and $\gd$ and
corresponding pre-orders.
These results suggest that, for at least the class of left cancellative
monoids, our thin triangle condition carries with it some genuinely
``hyperbolic''
structure. Truly geometric methods for left cancellative monoids (see for example
\cite{K_qsifp}) are of
considerable interest, for example because the word problem for one-relation
monoids (arguably the most significant open question in semigroup theory) is
reducible to the left cancellative case \cite{Adyan87}. The relationship between
one-relator presentations and hyperbolicity for groups has been studied in for
example \cite{Ivanov98}.

We have noted that previous notions of hyperbolicity for monoids appear to
be too weak, in the sense that they do not suffice
to control the behaviour of the monoid in other ways. In contrast,
our new definition yields a great deal of control on the monoid but may
perhaps be too strong, in that there are examples of ``well-behaved''
cancellative monoids which one might intuitively expect to be hyperbolic
but which do not satisfy this condition. It remains a very interesting and
important open question whether there is an intermediate notion which
encapsulates a wider class of cancellative monoids while still yielding
information comparable with that obtained for groups.

All of the above discussion applies to cancellative (and in some respects
left cancellative) monoids. In the more general case of monoids without
a cancellativity condition, we suspect there is no meaningful definition
of hyperbolicity. In
Section~\ref{sec_noncanc}, we give an elementary way to
construct finitely generated monoids (right cancellative but not
left cancellative) whose Cayley graphs
are essentially trees (and hence likely to be hyperbolic in any reasonable
geometric sense), but which need not even be recursively presentable. This
seems to be very strong evidence that no geometric hyperbolicity-type
condition on a Cayley graph will will impose much control on a general monoid.

\section{Directed Graphs and Thin Triangles}\label{sec_semimetric}

The main objects of study in this paper are directed graphs, which we allow to have loops and multiple directed edges. 
We shall view directed graphs 
chiefly as sets of vertices with an asymmetric, partially defined ``distance'' function
given by setting $d(u,v)$ to be the shortest length of a directed path from
$u$ to $v$ if there is such a path, or $\infty$ otherwise. (In particular, when viewed in this way, it is unimportant
whether the digraph has loops and/or multiple edges.) With this
convention a directed graph is a \textit{semimetric space} of the kind studied in
\cite{K_svarc,K_semimetric,K_qsifp}. 

In more detail, let $\rinf$ denote the set
$\mathbb{R}^{\geq 0} \cup \lbrace \infty \rbrace$ of
non-negative real numbers with $\infty$ adjoined. We equip it with the
obvious order, addition and multiplication, leaving $0 \infty$ undefined.
Now let $X$ be a set. A function $d : X \times X \to \rinf$ is called a
\emph{semimetric} on $X$ if:
\begin{itemize}
\item[(i)] $d(x,y)=0$ if and only if $x=y$; and
\item[(ii)] $d(x,z) \leq d(x,y) + d(y,z)$
\end{itemize}
for all $x,y,z \in X$. A set equipped with a semimetric on it is a
\textit{semimetric space}. In particular any directed graph is a semimetric space, where the distance between two vertices defined to be
the length of the shortest directed path between them, or $\infty$ if
there is no such path. 

If $X$ is a directed graph, the \textit{underlying undirected graph} of $X$ is the
graph with the same vertex set and an edge from $x$ to $y$ whenever there is an edge from
$x$ to $y$ or an edge from $y$ to $x$. 
In particular, if in the directed graph there are multiple directed edges connecting a given pair of vertices, then in the underlying undirected graph there will be a single undirected edge joining this pair of vertices. 

Now let $X$ be a directed graph (or semimetric
space), $x_0 \in X$ be a point, and $r$ be a positive real number or $\infty$.
The 
 \emph{out-ball} of radius $r$ based at $x_0$ is
\[
\overrightarrow{\mathcal{B}}_r(x_0) = \{ y \in X \mid d(x_0,y) \leq r \}.
\]
Dually, the \emph{in-ball} of radius $r$ based at $x_0$ is defined by
\[
\overleftarrow{\mathcal{B}}_r(x_0) = \{ y \in X \mid d(y,x_0) \leq r \},
\]
and the \emph{strong ball} of radius $r$ based at $x_0$ is
\[
\mathcal{B}_r(x_0) = \overrightarrow{\mathcal{B}}_r(x_0) \cap
 \overleftarrow{\mathcal{B}}_r(x_0).
\]
If $S$ is a set of points then $\overrightarrow{\mathcal{B}}_r(S)$,
$\overleftarrow{\mathcal{B}}_r(S)$ and $\mathcal{B}_r(x_0)$ are defined to be
the unions of the appropriate out-balls, in-balls and strong balls respectively
around points in $S$.

A directed graph (or semimetric space) is called \textit{$\delta$-bounded} if no
finite distance in the space exceeds $\delta$ (although there may be
points at distance $\infty$ in one or both directions). It is called
\textit{bounded} if it is $\delta$-bounded for some finite $\delta$.

For $n \geq 0$, a \textit{path} [of length
$n$ in a directed graph $X$] is a sequence $[x_0, \dots, x_n]$ of vertices
such that $X$ has
an edge from $x_i$ to $x_{i+1}$ for $0 \leq i < n$. The vertices $x_0$ and
$x_n$ are the \textit{start} and \textit{end} of the path, respectively, and
are denoted $\iota p$ and $\tau p$. The
path is called a \textit{geodesic} if $n = d(x_0, x_n)$. A path is called
\textit{simple} if it does not contain a repeated vertex. 

Now let $p = [x_0, \dots, x_n]$ and $q = [y_0, \dots, y_m]$ be paths in
$X$. We say that $p$ and $q$ are \textit{parallel} if $x_0 = y_0$ and
$x_n = y_m$, and we write $p \parallel q$. 
The paths $p$ and $q$ are 
\textit{composable} if $x_n = y_0$, in which case 
we define $p \circ q$ to be the path $[x_0, \dots, x_{n-1}, y_0, \dots, y_m]$.
A \textit{geodesic triangle} in $X$ is an ordered triple $(p, q, r)$ of
geodesics such that $p$ and $q$ are composable, and $p \circ q$ is parallel
to $r$.

\begin{definition}
Let $\delta$ be a non-negative real number.
A geodesic triangle in a directed graph is \textbf{$\delta$-thin} if
each side $x$ is contained in the union of the out-balls of radius $\delta$
around points on the side of the triangle meeting the start of $x$, together
with the in-balls of radius $\delta$ around points on the side of the triangle
meeting the end of $x$.
\end{definition}

\begin{figure}
\def\x{0.52}
\begin{center}
\begin{tikzpicture}[scale=0.55, decoration={
markings,
mark=
at position \x
with
{
\arrow[scale=1.5]{stealth}
}
}
]
\tikzstyle{vertex}=[circle,draw=black, fill=white, inner sep = 0.75mm]
\node (q)  at (1.4,4.7) {$q$};
\node (p)  at (-1.4,4.7) {$p$};
\node (B1)  at (7,8) {$\inball_\delta(q)$};
\node (B2)  at (-7,8) {$\outball_\delta(p)$};
\node (P1)  [vertex,label={180:{}}] at (-5,0) {};
\node (P2)  [vertex,label={180:{}}] at (5,0) {};
\node (P3)  [vertex,label={180:{}}] at (0,8.66) {};
\draw [postaction={decorate}] (P1) to [out=25,in=180-25] (P2);
\draw (P3) to [out=-60-25,in=120+25] (P2);
\draw (P3) to [out=240+25,in=60-25] (P1);
\draw [dashed, rotate=60] (2.5,4.33) ellipse (5.5cm and 4.3cm);
\draw [dashed, rotate=-60] (2.5-5,4.33) ellipse (5.5cm and 4.3cm);
\end{tikzpicture}
\end{center}
\caption{A schematic illustration of a $\delta$-thin directed geodesic
triangle in a directed graph: each side of the triangle is contained in the union of
the $\delta$-outball around the side that meets its initial vertex, and the
$\delta$-inball around the side that meets its terminal vertex.}\label{fig_defn}
\end{figure}
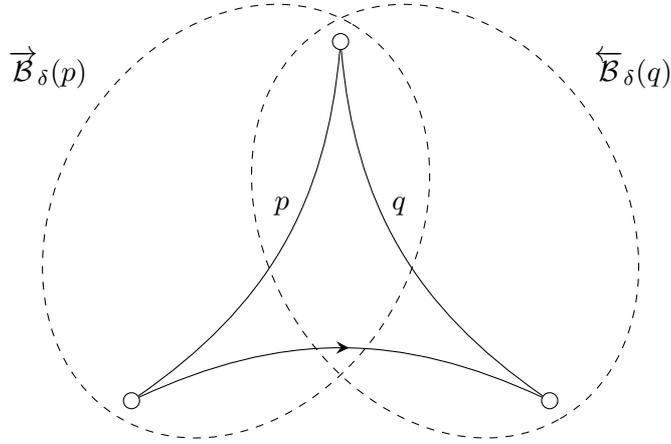

This definition is illustrated schematically in Figure~\ref{fig_defn}. Note that although
the three sides of a directed triangle play different roles, the definition
treats them in a uniform way. Made explicit, the definition says that a
directed geodesic triangle $(p,q,r)$ is $\delta$-thin if:
\begin{itemize}
\item every vertex in $r$ is contained within the union of the out-balls of
radius $\delta$ around vertices in $p$, and the in-balls of radius
$\delta$ around vertices in $q$; and
\item every vertex in $p$ is contained within the union of the out-balls of
radius $\delta$ around vertices in $r$, and the in-balls of radius
$\delta$ around vertices in $q$; and
\item every vertex in $q$ is contained within the union of the out-balls of
radius $\delta$ around vertices in $p$, and the in-balls of radius
$\delta$ around vertices in $r$.
\end{itemize}

\begin{definition}
A directed graph is called \textit{strongly $\delta$-hyperbolic} if all
of its geodesic triangles are $\delta$-thin.
\end{definition}

Given an undirected graph $Y$ we say that this graph is strongly $\delta$-hyperbolic if the digraph obtained by replacing each edge of $Y$ by a pair of oppositely oriented directed edges is strongly $\delta$-hyperbolic. With this definition an undirected graph is strongly $\delta$-hyperbolic if and only if it is $\delta$-hyperbolic in the classical sense (viewed as a metric space). 

The following propositions give some basic examples of strongly hyperbolic
directed graphs.

\begin{proposition}\label{prop_boundedspace}
Any $\delta$-bounded directed graph is strongly $\delta$-hyperbolic.
\end{proposition}
\begin{proof}
Since $\delta$ is an upper bound on the finite distances in the graph, it is an upper bound on the
length of geodesics in the graph. It follows that every vertex within a directed
geodesic triangle is contained with an out-ball of radius $\delta$ around its start
point, which suffices to show that every directed geodesic triangle is
$\delta$-thin.
\end{proof}

\begin{proposition}\label{prop_underlyingtree}
Let $X$ be a directed graph whose underlying undirected graph is a tree. Then $X$ is strongly
$0$-hyperbolic.
\end{proposition}
\begin{proof}
Let $Y$ be the undirected graph underlying $X$, and let $(p,q,r)$ be a directed geodesic triangle in $X$.
Then in particular $p$, $q$ and $r$ represent simple paths in $X$, and hence also in $Y$. Since $Y$ is
a tree and $p \circ q$ is parallel to $r$, it follows that there is a prefix $p'$ of $p$ and a suffix
$q'$ of $q$ such that $r = p' q'$. But then every vertex of $r$ lies on either $p$ or $q$. Similar
arguments show that every vertex of $p$ also lies on $q$ or $r$, and every vertex of $q$ also lies on
$p$ or $r$. This suffices to show that $X$ is strongly $0$-hyperbolic.
\end{proof}

Note however that it is \emph{not} the case in general that if the underlying
undirected graph of a digraph $X$ is hyperbolic then $X$ itself must be strongly hyperbolic. A counterexample may easily be constructed by taking a digraph $X$ that is not strongly hyperbolic, and considering the digraph $X^0$ obtained from $X$ by adding an extra vertex $z$ and an edge from
every vertex to $z$. Then by Proposition~\ref{prop_adjoinsink} below the digraph $X^0$ will not be strongly hyperbolic, while the underlying undirected graph of $X^0$ will be strongly hyperbolic, since it has bounded diameter. 

Associated with any semimetric space $X$ is a natural preorder $\lesssim$ relation given by $x \lesssim y$ if and only if $d(y,x) < \infty$. Let $\sim$ denote the equivalence relation given by $x \sim y$ if and only if
$x \lesssim y$ and $y \lesssim x$. We call the $\sim$-classes the \emph{strongly connected components} of $X$. 

\begin{proposition}\label{prop_component}
If $X$ is a strongly $\delta$-hyperbolic directed graph, then every strongly
connected component of $X$ is strongly $\delta$-hyperbolic.
\end{proposition}
\begin{proof}
Let $Y$ be a strongly connected component of a strongly $\delta$-hyperbolic
graph $X$. First note that since no directed paths between vertices of $Y$
pass outside $Y$, distances between vertices in $Y$ are the same in $Y$ as
in $X$. It follows that every directed geodesic triangle in $Y$ is also a
directed geodesic triangle in $X$, and hence is $\delta$-thin in $X$. Since
distances are the same in $X$ as in $Y$, it follows that every directed
geodesic triangle in $Y$ is $\delta$-thin in $Y$. Thus, $Y$ is strongly
$\delta$-hyperbolic.
\end{proof}

\begin{proposition}\label{prop_adjoinsink}
Let $X$ be a directed graph, and let $X^0$ be the directed graph obtained from $X$ by adding an extra
vertex $z$ and an edge from every vertex of $X^0$ to the vertex $z$. If $X$ is strongly $\delta$-hyperbolic then
$X^0$ is strongly
$\max(1, \delta)$-hyperbolic. Conversely, if $X^0$ is strongly $\delta$-hyperbolic
then $X$ is strongly $\delta$-hyperbolic.
\end{proposition}
\begin{proof}
Suppose $X$ is strongly $\delta$-hyperbolic, and let $(p,q,r)$ be a directed geodesic triangle in $X^0$. If
$(p,q,r)$ lies entirely in $X$ then it is $\delta$-thin in $X$, and hence in $X^0$. Otherwise, it contains
the vertex $z$, and since there are no non-loop edges out of $z$, it follows that the common end-point of $q$ and
$r$ is $z$. But since $q$ and $r$ are geodesics and there are edges from every vertex of $X$ to the vertex $z$, it follows
that $q$ and $r$ are sides of length $1$, and the only vertices on them are the vertices of the triangle. Thus,
every vertex on $q$ is either on $p$ or $r$, and every point on $r$ is either on $p$ or $q$. Finally,
every point on $p$ is contained in an in-ball of radius $1$ about $z$, which lies on $q$. Thus, the triangle
is $1$-thin. Hence, $X^0$ is strongly $\phi$-hyperbolic where $\phi = \max(1, \delta)$.

The converse follows from Proposition~\ref{prop_component}, since $X$
is a strongly connected component in $X^0$.
\end{proof}

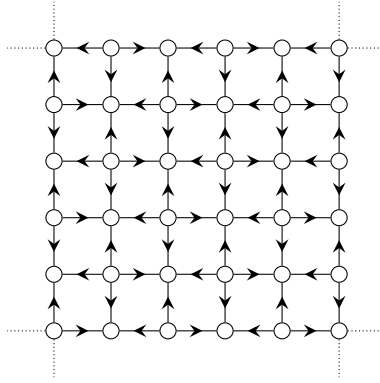
\begin{figure}
\def\x{0.65}
\begin{center}
\begin{tikzpicture}[scale=0.75, decoration={ 
markings,
mark=
at position \x
with 
{ 
\arrow[scale=1.5]{stealth} 
} 
} 
]
\tikzstyle{vertex}=[circle,draw=black, fill=white, inner sep = 0.75mm]
\node (Q1)  [label={180:{}}] at (0,-1) {};
\node (Q2)  [label={180:{}}] at (-1,0) {};
\node (Q3)  [label={180:{}}] at (0,6) {};
\node (Q4)  [label={180:{}}] at (-1,5) {};
\node (Q5)  [label={180:{}}] at (5,6) {};
\node (Q6)  [label={180:{}}] at (6,5) {};
\node (Q7)  [label={180:{}}] at (6,0) {};
\node (Q8)  [label={180:{}}] at (5,-1) {};
\node (P00)  [vertex,label={180:{}}] at (0,0) {};
\node (P10)  [vertex,label={180:{}}] at (1,0) {};
\node (P20)  [vertex,label={180:{}}] at (2,0) {};
\node (P30)  [vertex,label={180:{}}] at (3,0) {};
\node (P40)  [vertex,label={180:{}}] at (4,0) {};
\node (P50)  [vertex,label={180:{}}] at (5,0) {};
\node (P01)  [vertex,label={180:{}}] at (0,1) {};
\node (P11)  [vertex,label={180:{}}] at (1,1) {};
\node (P21)  [vertex,label={180:{}}] at (2,1) {};
\node (P31)  [vertex,label={180:{}}] at (3,1) {};
\node (P41)  [vertex,label={180:{}}] at (4,1) {};
\node (P51)  [vertex,label={180:{}}] at (5,1) {};
\node (P02)  [vertex,label={180:{}}] at (0,2) {};
\node (P12)  [vertex,label={180:{}}] at (1,2) {};
\node (P22)  [vertex,label={180:{}}] at (2,2) {};
\node (P32)  [vertex,label={180:{}}] at (3,2) {};
\node (P42)  [vertex,label={180:{}}] at (4,2) {};
\node (P52)  [vertex,label={180:{}}] at (5,2) {};
\node (P03)  [vertex,label={180:{}}] at (0,3) {};
\node (P13)  [vertex,label={180:{}}] at (1,3) {};
\node (P23)  [vertex,label={180:{}}] at (2,3) {};
\node (P33)  [vertex,label={180:{}}] at (3,3) {};
\node (P43)  [vertex,label={180:{}}] at (4,3) {};
\node (P53)  [vertex,label={180:{}}] at (5,3) {};
\node (P04)  [vertex,label={180:{}}] at (0,4) {};
\node (P14)  [vertex,label={180:{}}] at (1,4) {};
\node (P24)  [vertex,label={180:{}}] at (2,4) {};
\node (P34)  [vertex,label={180:{}}] at (3,4) {};
\node (P44)  [vertex,label={180:{}}] at (4,4) {};
\node (P54)  [vertex,label={180:{}}] at (5,4) {};
\node (P05)  [vertex,label={180:{}}] at (0,5) {};
\node (P15)  [vertex,label={180:{}}] at (1,5) {};
\node (P25)  [vertex,label={180:{}}] at (2,5) {};
\node (P35)  [vertex,label={180:{}}] at (3,5) {};
\node (P45)  [vertex,label={180:{}}] at (4,5) {};
\node (P55)  [vertex,label={180:{}}] at (5,5) {};
\arrowdraw{P00}{P01};
\arrowdraw{P02}{P03};
\arrowdraw{P04}{P05};
\arrowdraw{P00}{P10};
\arrowdraw{P20}{P30};
\arrowdraw{P40}{P50};
\arrowdrawb{P10}{P11};
\arrowdrawb{P12}{P13};
\arrowdrawb{P14}{P15};
\arrowdrawb{P01}{P11};
\arrowdrawb{P21}{P31};
\arrowdrawb{P41}{P51};
\arrowdraw{P20}{P21};
\arrowdraw{P22}{P23};
\arrowdraw{P24}{P25};
\arrowdraw{P02}{P12};
\arrowdraw{P22}{P32};
\arrowdraw{P42}{P52};
\arrowdrawb{P30}{P31};
\arrowdrawb{P32}{P33};
\arrowdrawb{P34}{P35};
\arrowdrawb{P03}{P13};
\arrowdrawb{P23}{P33};
\arrowdrawb{P43}{P53};
\arrowdraw{P40}{P41};
\arrowdraw{P42}{P43};
\arrowdraw{P44}{P45};
\arrowdraw{P04}{P14};
\arrowdraw{P24}{P34};
\arrowdraw{P44}{P54};
\arrowdrawb{P50}{P51};
\arrowdrawb{P52}{P53};
\arrowdrawb{P54}{P55};
\arrowdrawb{P05}{P15};
\arrowdrawb{P25}{P35};
\arrowdrawb{P45}{P55};
\arrowdraw{P11}{P12};
\arrowdraw{P13}{P14};
\arrowdraw{P11}{P21};
\arrowdraw{P31}{P41};
\arrowdrawb{P01}{P02};
\arrowdrawb{P03}{P04};
\arrowdrawb{P10}{P20};
\arrowdrawb{P30}{P40};
\arrowdraw{P31}{P32};
\arrowdraw{P33}{P34};
\arrowdraw{P13}{P23};
\arrowdraw{P33}{P43};
\arrowdrawb{P21}{P22};
\arrowdrawb{P23}{P24};
\arrowdrawb{P12}{P22};
\arrowdrawb{P32}{P42};
\arrowdraw{P51}{P52};
\arrowdraw{P53}{P54};
\arrowdraw{P15}{P25};
\arrowdraw{P35}{P45};
\arrowdrawb{P41}{P42};
\arrowdrawb{P43}{P44};
\arrowdrawb{P14}{P24};
\arrowdrawb{P34}{P44};
\dedgedraw{P00}{Q1}
\dedgedraw{P00}{Q2}
\dedgedraw{P05}{Q3}
\dedgedraw{P05}{Q4}
\dedgedraw{P55}{Q5}
\dedgedraw{P55}{Q6}
\dedgedraw{P50}{Q7}
\dedgedraw{P50}{Q8}
\end{tikzpicture}
\end{center}
\caption{An infinite directed grid, giving an example of a directed graph
that is strongly $\delta$-hyperbolic (since there are no directed geodesic
triangles)
but such that the underlying undirected graph is not $\delta$-hyperbolic.}
\end{figure}

\section{Triangle and Polygon Inequalities}

One of the difficulties of working with semimetric spaces is the limited
nature of the triangle inequality. If $(p,q,r)$ is a geodesic triangle then
certainly $|r| \leq |p| + |q|$, but we do not automatically have an upper
bound on $|p|$ in terms of $|q|$ and $|r|$, or on $|q|$ in terms of $|p|$
and $|r|$. In many cases of interest, however, hyperbolicity allows us to
acquire such bounds, albeit rather weaker than the conventional triangle
inequality.

Recall that a semimetric space is called \textit{quasi-metric} if there
is a constant $\lambda$ such that $d(x,y) \leq \lambda d(y,x) + \lambda$
for all points $x$ and $y$. Equivalently, a space is quasi-metric if it
is quasi-isometric to a metric space \cite{K_semimetric}.

\begin{lemma}\label{lemma_quasitriangle}
Let $X$ be a locally finite, strongly $\delta$-hyperbolic directed graph with indegree
and outdegree bounded by $\alpha$. Then there is a constant $\lambda$,
depending only on (and polynomial-time computable from) $\delta$ and $\alpha$, such
that whenever $p$, $q$, and $r$ are the sides of a directed geodesic
triangle in $X$ (in no particular order), we have
\[
|p| \le \lambda (|q| + |r|).
\]
\end{lemma}
\begin{proof}
If $p$ is the hypotenuse the result follows from the standard triangle inequality for semimetric spaces. 
Now suppose that $r$ is the hypotenuse and that $\tau p = \iota q$ (the third remaining case being dual
to this one). 
If $|r|=|q|=0$ then, since $p$ is geodesic, it follows that $|p|=0$ and the result holds. So we may now suppose that either $|r|\geq 1$ or $|q| \geq 1$. Since $p$ is geodesic it has no repeated vertices. Also since $X$ is strongly $\delta$-hyperbolic it follows that the
vertices of $p$ are all contained in the set
$
\outball_{\delta}(r) \cup \inball_{\delta}(q).
$
We may assume that the graph $X$ has at least one directed edge, since otherwise the result is trivially true, and so in particular $\alpha \geq 1$.  
As the outdegree is bounded by $\alpha$ it follows that
the outdegree is also bounded by $\alpha+1$ and so
\[
| \outball_{\delta}(r) | \le 
(| r | + 1)
\sum_{i=0}^{\delta}((\alpha+1)^i)
=
(| r | + 1)
\left( \frac{(\alpha+1)^{\delta+1} - 1}{(\alpha+1)-1} \right).
\]
(We take $\alpha+1$ rather than $\alpha$ here just to avoid having to deal with the case $\alpha=1$ seperately.)
Also, the indegree is bounded by $\alpha$, so one obtains a similar bound for $| \inball_{\delta}(q) |$, which combined with the above inequality gives
\[
|p| \le 
(| r | + 1 + |q| + 1)
\left( \frac{(\alpha+1)^{\delta+1} - 1}{(\alpha+1)-1} \right)
\le
\left(
\frac{3((\alpha+1)^{\delta+1} - 1)}{(\alpha+1)-1} 
\right) (|q| + |r|),
\]
completing the proof. 
\end{proof}

\begin{corollary}
\label{cor_poly}
Let $X$ be a locally finite strongly $\delta$-hyperbolic directed graph with indegree
and outdegree bounded by $\alpha$. Then there is a constant $\lambda$,
depending only on (and polynomial-time computable from) $\delta$ and
$\alpha$, such that every strongly connected component of $X$ is $(\lambda,0)$-quasi-metric.
\end{corollary}
\begin{proof}
Let $\lambda$ be the constant given by Lemma~\ref{lemma_quasitriangle}.
Now suppose $x$ and $y$ belong to the same strongly connected component
of $X$. Let $p$ be a geodesic from $x$ to $y$, $q$ a geodesic from
$y$ to $x$, and $e$ the (geodesic) empty path at $x$. Then the triple
$(x,y,e)$ is a geodesic triangle, and so Lemma~\ref{lemma_quasitriangle}
yields
\[
d(x,y) = |p| \leq \lambda (|q| + |e|) = \lambda (d(y,x) + 0) = \lambda d(y,x).
\] 
\end{proof}

\begin{definition}[Directed geodesic $n$-gon]
A directed geodesic $n$-gon in a directed graph $\Gamma$ is an $n$-tuple $(p_1,\ldots,p_{n-1},q)$ of geodesic paths such that $p = p_1 \circ p_2 \circ \cdots \circ p_{n-1}$ is defined, and $p \parallel q$.  
\end{definition}

\begin{theorem}[Polygon quasi-inequality]
\label{thm_poly}
\begin{sloppypar}
Let $X$ be a locally finite strongly $\delta$-hyperbolic directed graph with
indegree and outdegree bounded by $\alpha$. Then there is a constant
$K$, depending only on (and polynomial-time computable from) $\alpha$
and $\delta$, such that every side
length of a directed geodesic $n$-gon is bounded above by $K$ times the sum
of the other side lengths.
\end{sloppypar} 
\end{theorem}
\begin{proof}
Let $\lambda$ be the constant given by Lemma~\ref{lemma_quasitriangle}
and let $K = \max(\lambda^2,1)$.

The case $n=3$ is immediate from Lemma~\ref{lemma_quasitriangle}.
Let $(p_{1}, \ldots, p_{n-1}, q)$ be a directed geodesic $n$-gon. Since
$q$ is geodesic it is
immediate from the triangle inequality that
$|q| \leq |p_1| + \dots + |p_n| \leq K (|p_1| + \dots + |p_n|)$.
It remains only to prove that the sides $p_i$ satisfy the claimed
bound when $n \geq 4$.

We treat first the case where $n=4$ and $i = 2$. (This case is special
because our general strategy will involve choosing two composable sides of
the polygon excluding $p_i$, and this is the only case
where this approach is impossible.)
In this case, let $r$ be a
geodesic from $\iota p_1$ to $\tau p_2$. Then $(r, p_3, q)$ and
$(p_1, p_2, r)$ are geodesic triangles, so applying Lemma~\ref{lemma_quasitriangle}
twice we have
$|r| \leq \lambda (|p_3| + |q|)$ and
$$|p_2| \leq \lambda (|p_1| + |r|) \leq \lambda (|p_1| + \lambda (|p_3| + |q|)) \leq K (|p_1| + |p_3| + |q|)$$
as required.

We establish the result in the remaining cases by induction on $n$.
Consider, then, a geodesic $n$-gon $(p_1,\ldots, p_{n-1}, q)$ where
either $n > 4$, or $n = 4$ but $i \neq 2$, and suppose the claim holds
for all geodesic $(n-1)$-gons.

Suppose first that $i > 2$. Let $r$ be a geodesic from $\iota p_{1}$ to
$\tau p_{2}$. Then by the
triangle inequality $|r| \leq |p_1| + |p_2|$. Now
$(r, p_2, \ldots, p_{n-1}, q)$ is a geodesic $(n-1)$-gon, so
by the inductive hypothesis we have
\begin{eqnarray*}
|p_{i}|
& \le & K \left( | r | + | p_{3} | + \cdots + | p_{i-1} | + | p_{i+1} | + \cdots + | p_{n-1} | + | q | \right) \\
& \le & K \left( (|p_{1}| + |p_{2}|) + | p_{3} | + \cdots + | p_{i-1} | + | p_{i+1} | + \cdots + | p_{n-1} | + | q | \right).
\end{eqnarray*}

Finally suppose $i \leq 2$. By assumption, either $n > 4$, or else $n = 4$
but $i \neq 2$. It follows that we have $i \neq p_{n-1}$ and $i \neq p_{n-2}$.
Now we can apply the same argument as in the previous case, but taking $r$
this time to be a geodesic from $\iota p_{n-1}$ to $\tau p_n$.
\end{proof}

We note that Lemma~\ref{lemma_quasitriangle} and Theorem~\ref{thm_poly} can
both fail if the hypothesis of local finiteness is dropped. Indeed, let $X$
be the directed graph with vertex set $\mathbb{Z}$ and an edge from $i$ to
$i+1$ for each $i \in \mathbb{Z}$, and let $Y = X^0$ be the graph obtained
from $X$ by the construction in Proposition~\ref{prop_adjoinsink}. Then
$Y$ is strongly $1$-hyperbolic by Propositions~\ref{prop_underlyingtree}
and \ref{prop_adjoinsink}, but contains
directed geodesic triangles with two sides
of length $1$ and the third arbitrarily long.

\section{Monoids and Cayley Graphs}\label{sec_hypmonoids}

Let $M$ be a monoid generated by a finite set $S$. Then $M$ naturally
has the structure of a directed graph (with vertex set $M$, and an edge
from $m$ to $n$ for each $s \in S$ such that $ms = n$), and hence also
of a semimetric space. This directed graph is called the \textit{right Cayley graph}
of $M$.
\begin{definition}
A monoid $M$ generated by a finite subset $S$ is called \textit{strongly
$\delta$-hyperbolic} (with respect to the generating set $S$) if its right
Cayley graph with respect to $S$ is strongly $\delta$-hyperbolic.
A monoid $M$ is called \textit{strongly hyperbolic} if it is strongly $\delta$-hyperbolic
for some $\delta$ with respect to some finite generating set.
\end{definition}

We do not presently know whether a strongly hyperbolic
monoid is necessarily strongly hyperbolic with respect to every choice
of finite generating set. This question is deserving of further study.

Our initial results about strongly $\delta$-hyperbolic spaces immediately
give a number of examples of monoids which are strongly hyperbolic.

\begin{example}
Every finite monoid is strongly $\delta$-hyperbolic with respect to every generating set, where
$\delta$ is the length of the longest geodesic representative for an element of $M$. Indeed,
it is easy to see that the Cayley graph is $\delta$-bounded, so this follows from
Proposition~\ref{prop_boundedspace}.
\end{example}

\begin{example}
Free monoids of finite rank are strongly $0$-hyperbolic with respect to their free generating sets. Indeed,
the underlying undirected graph of the Cayley graph is a tree, so this follows from Proposition~\ref{prop_underlyingtree}.
\end{example}

\begin{example}
The bicyclic monoid $\mathcal{B} = \langle p, q \mid pq = 1 \rangle$ is strongly $0$-hyperbolic with respect to the standard
generating set $\lbrace p, q \rbrace$. Again, the underlying undirected graph of the Cayley graph (see
Figure~\ref{fig_bicyclic}) is a tree, so this follows from Proposition~\ref{prop_underlyingtree}.
\end{example}

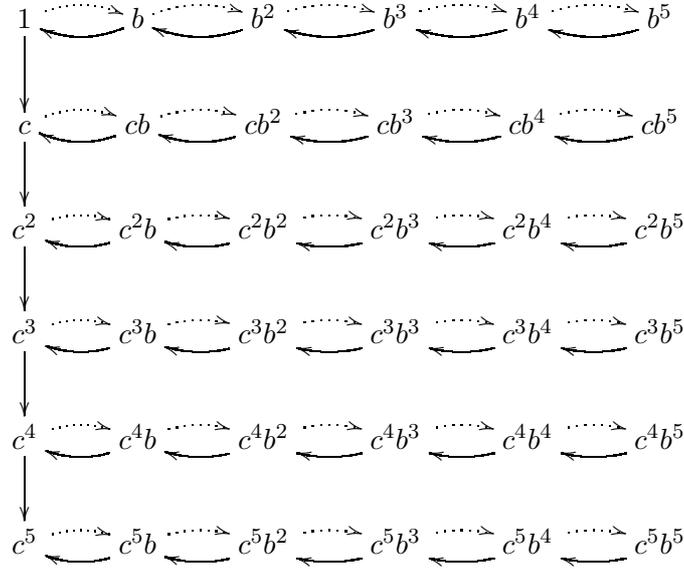
\begin{figure}
\begin{equation*}
\xy
\xymatrix{ 
1 \ar@{->}[d] \ar@{.>}@/^/[r] \ar@{<-}@/_/[r] & 
b \ar@{.>}@/^/[r] \ar@{<-}@/_/[r] & {b^2}  \ar@{.>}@/^/[r] \ar@{<-}@/_/[r] & {b^3}  \ar@{.>}@/^/[r] \ar@{<-}@/_/[r] & {b^4}  \ar@{.>}@/^/[r] \ar@{<-}@/_/[r] & {b^5}      \\
c \ar@{->}[d] \ar@{.>}@/^/[r] \ar@{<-}@/_/[r] & 
cb \ar@{.>}@/^/[r] \ar@{<-}@/_/[r] & {cb^2}  \ar@{.>}@/^/[r] \ar@{<-}@/_/[r] & {cb^3}  \ar@{.>}@/^/[r] \ar@{<-}@/_/[r] & {cb^4}  \ar@{.>}@/^/[r] \ar@{<-}@/_/[r] & {cb^5}      \\
c^2 \ar@{->}[d] \ar@{.>}@/^/[r] \ar@{<-}@/_/[r] & 
c^2b \ar@{.>}@/^/[r] \ar@{<-}@/_/[r] & {c^2b^2}  \ar@{.>}@/^/[r] \ar@{<-}@/_/[r] & {c^2b^3}  \ar@{.>}@/^/[r] \ar@{<-}@/_/[r] & {c^2b^4}  \ar@{.>}@/^/[r] \ar@{<-}@/_/[r] & {c^2b^5}      \\
c^3 \ar@{->}[d] \ar@{.>}@/^/[r] \ar@{<-}@/_/[r] & 
c^3b \ar@{.>}@/^/[r] \ar@{<-}@/_/[r] & {c^3b^2}  \ar@{.>}@/^/[r] \ar@{<-}@/_/[r] & {c^3b^3}  \ar@{.>}@/^/[r] \ar@{<-}@/_/[r] & {c^3b^4}  \ar@{.>}@/^/[r] \ar@{<-}@/_/[r] & {c^3b^5}      \\
c^4 \ar@{->}[d] \ar@{.>}@/^/[r] \ar@{<-}@/_/[r] & 
c^4b \ar@{.>}@/^/[r] \ar@{<-}@/_/[r] & {c^4b^2}  \ar@{.>}@/^/[r] \ar@{<-}@/_/[r] & {c^4b^3}  \ar@{.>}@/^/[r] \ar@{<-}@/_/[r] & {c^4b^4}  \ar@{.>}@/^/[r] \ar@{<-}@/_/[r] & {c^4b^5}      \\
c^5  \ar@{.>}@/^/[r] \ar@{<-}@/_/[r] & 
c^5b \ar@{.>}@/^/[r] \ar@{<-}@/_/[r] & {c^5b^2}  \ar@{.>}@/^/[r] \ar@{<-}@/_/[r] & {c^5b^3}  \ar@{.>}@/^/[r] \ar@{<-}@/_/[r] & {c^5b^4}  \ar@{.>}@/^/[r] \ar@{<-}@/_/[r] & {c^5b^5}      
}
\endxy
\end{equation*}
\caption{A partial view of the right Cayley graph of the bicyclic monoid $\langle b,c \; | \; bc=1 \rangle$ where 
$\rightarrow$ corresponds to multiplication by $c$ and $\dashrightarrow$ corresponds to multiplication by $b$.}
\label{fig_bicyclic}
\end{figure}

\begin{example}
Polycyclic monoids of finite rank are strongly $1$-hyperbolic with respect to their standard generating
sets. Recall that for $n \geq 2$ the polycyclic monoid $\mathcal{P}_n$ of rank $n$ is given by the presentation
$$\langle p_1, \dots, p_n, q_1, \dots, q_n, z \mid p_i q_i = 1, p_i q_j = z = p_i z = q_i z = z p_i = z q_i \textrm{ for all } i \neq j \rangle.$$
The generator $z$ represents a zero element. Let $Y$ be the graph obtained from
the Cayley graph of $\mathcal{P}_n$ (with respect to the generating set from the presentation above) by removing the vertex $z$ and all edges incident with
it. Then it is straightforward to verify that the underlying undirected graph
is a tree, and hence is strongly $0$-hyperbolic by
Proposition~\ref{prop_underlyingtree}. Moreover, the Cayley graph of $\mathcal{P}_n$
can be recovered from $Y$ by the construction in Proposition~\ref{prop_adjoinsink},
so by that proposition $\mathcal{P}_n$ itself is strongly $1$-hyperbolic.
\end{example}

\begin{example}
Word hyperbolic groups are strongly hyperbolic in our sense. Indeed, suppose
$G$ is word hyperbolic and choose a finite generating set $S$ for $G$ which
is closed under the taking of inverses. Then $S$ is also a monoid generating
set for $G$, and the distance function on the monoid Cayley graph is the
same as that on the group Cayley graph. The claim now follows from the usual
``thin triangle'' property of hyperbolic metric spaces.
\end{example}

We can also expand our class of examples by showing closure under some
elementary semigroup-theoretic constructions.

\begin{proposition}\label{prop_adjoinzero}
Let $M$ be a monoid, and let $M^0$ be the monoid obtained from $M$ by
adjoining a new element $0$ which acts as a zero element. Then $M$ is
strongly hyperbolic if and only if $M^0$ is strongly hyperbolic.
\end{proposition}
\begin{proof}
Suppose $M$ is strongly hyperbolic with respect to a generating set $S$.
Then $M^0$
is generated by $S \cup \lbrace 0 \rbrace$. The Cayley graph of $M^0$ with
respect
to this generating set is
clearly obtained from that of $M$ by the construction in
Proposition~\ref{prop_adjoinsink}, and so by the proposition $M^0$ is strongly hyperbolic.

Conversely, suppose $M^0$ is strongly hyperbolic with respect to a generating set
$T$. Since $M^0 \setminus \lbrace 0 \rbrace$ is a subsemigroup of $M^0$,
we must have $0 \in T$. Now $T \setminus \lbrace 0 \rbrace$ is a generating
set for $M$, and just as above the Cayley graph for $M^0$ with respect to $T$ is again
obtained from that for $M$ with respect to $T \setminus \lbrace 0 \rbrace$
by the construction in Proposition~\ref{prop_adjoinsink}. Thus, by the proposition,
$M$ is strongly hyperbolic
\end{proof}

\begin{proposition}\label{prop_reesquotient}
Let $M$ be a finitely generated monoid and let $I$ be an ideal of $M$.
If $M$ is strongly $\delta$-hyperbolic then the Rees quotient $M/I$ is
strongly hyperbolic. 
\end{proposition}
\begin{proof}
Suppose that $M$ is strongly hyperbolic with respect to a finite generating
set $S$, and let $0$ be the $0$ element in the Rees quotient $M/I$. Then
$M / I$ is generated by the set $A = (S \cap (M \setminus I)) \cup \lbrace 0 \rbrace$. Let $X$ be the Cayley graph of $M / I$ with respect to $A$,
and $Y$ the Cayley graph of $M$ with respect to $S$. Let $x_0 \in X$ be the
vertex of $X$ corresponding to the zero element $0$ of $M / I$. Notice that
for any two non-$0$ elements of $M / I$ (that is, elements of $M \setminus I$),
we have $d_X(a,b) = d_Y(a,b)$.

We claim that $X$ is strongly $\max(1,\delta)$-hyperbolic. Let $(p,q,r)$
be a directed geodesic triangle in $X$. If $(p,q,r)$ does not involve the
vertex $0$ then the distances between vertices visited are the same in $X$ as
in $Y$; but $Y$ is strongly $\delta$-hyperbolic, so $(p,q,r)$ is $\delta$-thin in $Y$
and hence in $X$. Otherwise, $(p,q,r)$ contains the vertex $0$ and, 
arguing as in the proof of Proposition~\ref{prop_adjoinsink},
since in $X$ there are no edges out of $0$, it follows that the common
endpoint of $q$ and $r$ is $0$. But every point of $(p,q,r)$ has a
directed geodesic of length at most $1$ to $0$. It follows that $(p,q,r)$
is $1$-thin, completing the proof. 
\end{proof}

\begin{corollary}\label{cor_subsemigroup}
Let $M$ be a strongly hyperbolic monoid, and $S$ a submonoid which is
the complement of an ideal. Then $S$ is strongly hyperbolic.
\end{corollary}
\begin{proof}
Let $I = M \setminus S$. Then the Rees quotient $M / I$ is isomorphic
to $S^0$, so the result follows from Propostions~\ref{prop_adjoinzero}
and \ref{prop_reesquotient}.
\end{proof}

Recall that an element $x$ of a monoid is called a \textit{unit} if there
is an element $y$ such that $xy = yx = 1$; the set of all units forms a
(maximal) subgroup of $M$.
\begin{corollary}
The group of units of a cancellative strongly hyperbolic monoid is a hyperbolic group.  
\end{corollary}
\begin{proof}
It is well-known and easy to prove that the complement of the group
of units in a cancellative monoid forms an ideal, so this follows
from Corollary~\ref{cor_subsemigroup}.
\end{proof}
Of course, the converse to the latter corollary does not hold in general:
for example, the free commutative monoid of rank two is cancellative
with trivial (hence hyperbolic) group of units, but is itself not a strongly
hyperbolic monoid.

\section{Directed $2$-Complexes, Presentations and Dehn Functions}\label{sec_graphs}

It is well known that, even if a monoid is given by a finite 
presentation, the word problem for the monoid may be undecidable. Markov 
\cite{Markov47} and Post \cite{Post47} proved independently that the word
problem for finitely 
presented monoids is undecidable in general; this result was extended
by Turing \cite{Turing50} to cancellative semigroups, and then by 
Novikov and Boone to groups (see \cite{LyndonandSchupp} for references).
For classes of monoids that do 
have decidable word problem it is natural to consider the complexity of 
the word problem. For example, monoids which admit presentations by 
finite complete rewriting systems have solvable word problem, but there 
is no bound on the complexity of the word problem for such monoids; see 
\cite{BauerOtto84}. On the other hand, automatic monoids have word problem that is solvable in 
quadratic time \cite[Corollary 3.7]{campbell_autsg}. As mentioned in the 
introduction, a finitely presented group is word hyperbolic exactly if 
it has linear Dehn function; recent results of Cain \cite{CainArXiV} 
show that word-hyperbolic semigroups (in the sense of Duncan and Gilman 
\cite{Duncan04}) have word problem solvable in polynomial time.
A remarkable result of Birget \cite{Birget98} characterises finitely generated
semigroups with word problem in $\mathcal{NP}$ as exactly those embeddable in
finitely presented semigroups with polynomial Dehn function. (An
analogous statement for groups was proved later \cite{BirgetAnn, SapirAnn}.)

Our main aim in this section is to prove that finitely generated, left
cancellative, strongly $\delta$-hyperbolic monoids are finitely presented with
polynomial Dehn functions, and therefore admit non-deterministic polynomial-time
word problem solutions. Our proof is most easily and intuitively expressed
in the language of \textit{direct $2$-complexes} \cite{K_qsifp,Guba06},
so we begin by briefly recalling some definitions and results concerning these.

Let $P(\Gamma)$ denote the set of all directed paths in a directed graph
$\Gamma$, including empty paths at each vertex. A {\em directed $2$-complex}
is a directed graph $\Gamma$ equipped with a set $F$ (called the {\em set
of $2$-cells}), and three maps $\topp{\cdot}\colon F \to P(\Gamma)$,
$\bott{\cdot}\colon F \to P(\Gamma)$, and $^{-1}\colon F \to F$ called {\em top},
{\em bottom}, and {\em inverse} such that
\begin{itemize}
\item for every $f\in F$, the paths $\topp{f}$ and $\bott{f}$ are
parallel; 
\item $^{-1}$ is an involution without fixed points, and
$\topp{f^{-1}}=\bott{f}$, $\bott{f^{-1}}=\topp{f}$ for every $f\in F$.
\end{itemize}

If $K$ is a directed $2$-complex, then the directed paths on $K$ are called
{\em $1$-paths}. For every $2$-cell $f\in F$, the vertices
$\iota(\topp{f})=\iota(\bott{f})$ and $\tau(\topp{f})=\tau(\bott{f})$ are
denoted $\iota(f)$ and $\tau(f)$, respectively.

An \emph{atomic $2$-path} is a triple $(p,f,q)$, where $p$, $q$ are $1$-paths in $K$, and $f \in F$ such that $\tau(p) = \iota(f)$, $\tau(f) = \iota(q)$. If $\delta$ is an atomic $2$-path then we use $\topp{\delta}$ to denote $p \topp{f} q$ and $\bott{\delta}$ is denoted by $p \bott{f} q$, these are the top and bottom $1$-paths of the atomic $2$-path. A \emph{non-trivial} $2$-path $\delta$ in $K$ is then a sequence of atomic paths $\delta_1$, $\ldots$, $\delta_n$, where $\bott{\delta_i} = \topp{\delta_{i+1}}$ for every $1 \leq i < n$, and the length of this $2$-path is $n$. The top and bottom $1$-paths of $\delta$, denoted $\topp{\delta}$ and $\bott{\delta}$ are then defined as $\topp{\delta_1}$ and $\bott{\delta_n}$, respectively. 

We use $\delta \circ \delta'$ to denote the composition of two $2$-paths. We say that $1$-paths $p$, $q$ in $K$ are homotopic if there exists a $2$-path $\delta$ such that $\topp{\delta} = p$ and $\bott{\delta} = q$. 
Recall that a pair of paths $p,q \in P(\Gamma)$ are said to be parallel, written $p \parallel q$, if $\iota p = \iota q$ and $\tau p = \tau q$. 
We say that a directed $2$-complex $K$ is \emph{directed simply connected} if for every pair of parallel paths $p \parallel q$, $p$ and $q$ are homotopic in $K$. 

Let $K$ be a directed $2$-complex with underlying directed graph $\Gamma$
and set of $2$-cells $F$, and let $T = (p,q,r)$ be a directed triangle in
$\Gamma$. Now let $K'$ be the $2$-complex obtained from $K$ by adjoining
one new element $f$ to $F$ satisfying $\topp{f} = p \circ q$ and
$\bott{f} = r$. We call $K'$ the directed $2$-complex obtained from $K$ by
adjoining a $2$-cell for the triangle $T$.  

\begin{definition}[Tessellation]
Given a pair of parallel paths $p$ and $q$ in a directed graph, we say
that a set $T_1, T_2, \ldots, T_r$ of directed geodesic triangles
\emph{tessellates} $p$ and $q$ if in the $2$-complex $K$ obtained by
adjoining $2$-cells for each $T_i$ we have $p \sim_K q$. We say
that a set of geodesic triangles \emph{tessellates} a directed
$n$-gon $(p_1, \dots, p_{n-1}, r)$ if it tessellates the paths
$p_1 \circ \dots \circ p_{n-1}$ and $r$.
\end{definition}

\begin{definition}
Given a directed triangle $T = (p,q,r)$ in a directed graph, we define the \textit{size} $\Sigma(T)$ of
$T$ to be $|p| + |q|$.  
\end{definition}

Our strategy for establishing our main result is to show that in a
strongly $\delta$-hyperbolic directed graph we can tessellate the ``gap''
between two parallel paths with (polynomially many, as a function of the
path lengths) geodesic triangles of bounded size. We begin by showing
that every pair of parallel paths can be tessellated by geodesic triangles
(of not necessarily bounded size).

\begin{lemma}
\label{lem_FillingFish}
Let $\Gamma$ be a directed graph. 
Every pair of parallel paths $p \parallel q$ in $\Gamma$ can be tessellated
by $|p| + |q| + 1$ directed geodesic triangles of size at most $2(|p| + |q|)$.
\end{lemma}
\begin{proof}
We prove the result by induction on $|p|+|q|$. The base case is trivial
since any path of length $0$ or $1$ is automatically geodesic. For the
induction step, if $p$ and $q$ are both geodesic then $(p,q)$ itself
naturally may be viewed as a single directed geodesic triangle of the
required size and we
are done. Now suppose that $p$, say, is not geodesic. Decompose
$p = p_\iota \circ p_e \circ p_\tau$ where $p_\iota \circ p_e$ is the shortest non-geodesic
subpath of $p$ and $p_e$ is a directed edge. Let $p'$ be a geodesic path
from $\iota p$ to $\tau p_e$. Then $|p' \circ p_{\tau}| < |p|$ and by induction the pair
$(p' \circ p_\tau, q)$ may be tessellated by $|p' \circ p_\tau| + |q| + 1$
directed geodesic triangles of size at most $2(|p' \circ p_{\tau}| + |q|) < 2(|p|+|q|)$.
Taken together with the directed geodesic
triangle $(p_i, p_e, p')$ which also has size less than $2(|p|+|q|)$ we conclude that $(p,q)$ may be tessellated by 
\[
(|p' \circ p_\tau| + 1) + |q| + 1
\leq
|p| + |q| + 1
\] 
directed geodesic triangles of size at most $2(|p|+|q|)$.
\end{proof}

\begin{lemma}
\label{lem_tessellating}
\begin{sloppypar}
Let $\Gamma$ be a strongly $\delta$-hyperbolic directed graph. Then 
every directed geodesic triangle $T$ can be tessellated by five
directed geodesic triangles (some of which may be trivial triangles
with a single vertex) of size no more than $\frac{3}{4} \Sigma(T) + 2\delta + 1$.
\end{sloppypar}
\end{lemma}
\begin{proof}
For clarity in this proof, we will use the convention that $XY$ denotes a geodesic path from a vertex
$X$ to a vertex $Y$, while $XYZ$ denotes the directed geodesic triangle $(XY,YZ,XZ)$. (Of course geodesics
are not unique; we will be careful to make clear where the choice is important.)

Let $T = PQR = (PQ, QR, PR)$ be a directed geodesic triangle.
Let $|PQ|=k$, $|QR|=l$ so that $\Sigma(T) = k+l$. Suppose $k \geq l$ (the case $l \geq k$ being dual); it
follows in particular that $l + \frac{1}{2} k \leq \frac{3}{4} (l+k)$.
Let $M$ be the vertex on the geodesic $PQ$ satisfying $d(P,M) = \bott{\frac{k}{2}}$ and
$d(M,Q) = \topp{\frac{k}{2}}$. Since $\Gamma$ is strongly $\delta$-hyperbolic there are now two cases to consider. 

\

\noindent \textbf{\boldmath Case (a): $\outball_\delta(M)$ intersects $QR$.} 
Let $O$ be a point in $\outball_\delta(M)$ which lies on $QR$. Consider geodesics $MO$, $PO$, $PM$,
$MQ$, $QO$ and $OR$, chosen so that $QO \circ OR = QR$ and $PM \circ MQ = PQ$. Consider also the three geodesic
triangles: $T_1 = MQO$, $T_2 = PMO$ and $T_3 = POR$ (see Figure~\ref{fig_casea}). The sizes of these
triangles are bounded as follows:
\begin{eqnarray*}
\Sigma(T_1) & \leq & \topp{\frac{k}{2}} + l \leq \frac{3}{4}(k+l) + 1 = \frac{3}{4} \Sigma(T) + 1, \\
\Sigma(T_2) & \leq & \bott{\frac{k}{2}} + \delta \leq \frac{1}{2} \Sigma(T) + \delta, \\
\Sigma(T_3) & \leq & \bott{\frac{k}{2}} + \delta + l \leq \frac{3}{4} \Sigma(T) + \delta. 
\end{eqnarray*}
So in this case, our triangle is tesselated by the three triangles $T_1$, $T_2$ and $T_3$, which satisfy
the required size bound.

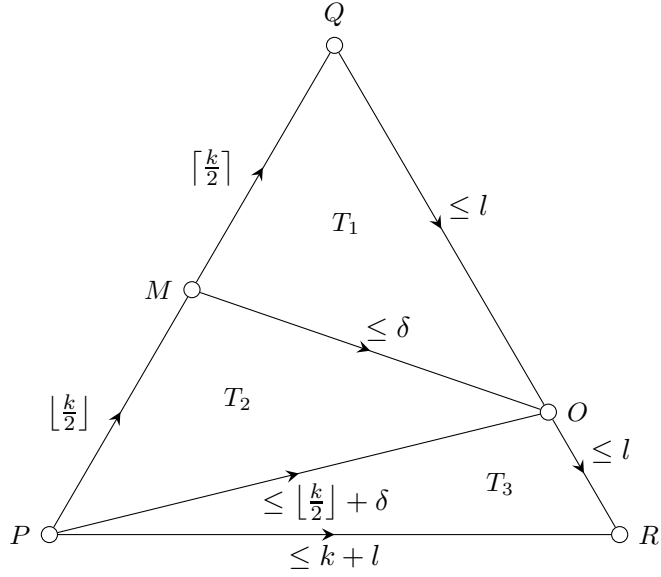
\begin{figure}
\def\x{0.5}
\begin{center}
\begin{tikzpicture}[scale=0.75, decoration={ 
markings,
mark=
at position \x
with 
{ 
\arrow[scale=1.5]{stealth} 
} 
} 
]
\tikzstyle{vertex}=[circle,draw=black, fill=white, inner sep = 0.75mm]
\node (P)  [vertex,label={180:{\small $P$}}] at (0,0) {};
\node (Q)  [vertex,label={90:{\small $Q$}}] at (5,8.65) {};
\node (R)  [vertex,label={360:{\small $R$}}] at (10,0) {};
\node (M)  [vertex,label={180:{\small $M$}}] at (2.5,4.33) {};
\node (O)  [vertex,label={360:{\small $O$}}] at (8.75,2.165) {};
\node (T_1)  [label={360:{\small $T_1$}}] at (4.6,5.5) {};
\node (T_2)  [label={360:{\small $T_2$}}] at (2.7,2.4) {};
\node (T_3)  [label={360:{\small $T_3$}}] at (7.3,0.9) {};
\draw [postaction={decorate}] (P)--(M) node[pos=0.5,left] {$\bott{\frac{k}{2}}$ \; };
\draw [postaction={decorate}] (M)--(Q) node[pos=0.5,left] {$\topp{\frac{k}{2}}$ \; };
\draw [postaction={decorate}] (M)--(O) node[pos=0.5,above] {\; \; $\leq \delta$};
\draw [postaction={decorate}] (P)--(O) node[pos=0.5,below] {\; \; \; $\leq \bott{\frac{k}{2}} + \delta$};
\draw [postaction={decorate}] (Q)--(O) node[pos=0.5,above] {\; \; \; $\le l$};
\draw [postaction={decorate}] (O)--(R) node[pos=0.5,above] {\; \; \; $\le l$};
\draw [postaction={decorate}] (P)--(R) node[pos=0.5,below] {$\le k+l$};
%
\end{tikzpicture}
\end{center}
\caption{Proof of Lemma~\ref{lem_tessellating}, Case (a).}
\label{fig_casea}
\end{figure}

\noindent \textbf{\boldmath Case (b): $\inball_\delta(M)$ intersects $PR$.} Let $O$ be a point in
$\inball_\delta(M)$ which lies on $PR$. Consider geodesics $OM$, $OQ$, $PM$, $MQ$, $PO$ and $OR$
such that $PM \circ MQ = PQ$ and $PO \circ OR = PR$. Consider also the three geodesic triangles
$T_1 = POM$, $T_2=OMQ$ and $T_3=OQR$
(see Figure~\ref{fig_caseb}). The triangles, $T_2$ and $T_3$ have size bounds:
\begin{eqnarray*}
\Sigma(T_2) & \leq & \delta + \topp{\frac{k}{2}} \leq \frac{1}{2} \Sigma(T) + \delta + 1, \\
\Sigma(T_3) & \leq & \topp{\frac{k}{2}} + \delta + l \leq \frac{3}{4}(k+l) + \delta + 1 = \frac{3}{4} \Sigma(T) + \delta + 1, \\
\end{eqnarray*}
It is not immediately clear how to get a such a size bound on $T_1$, so we further subdivide it in the middle of the edge $PO$,
at the point $U$. 
So set $x = |PO|$ and let $U$ be the vertex on the geodesic $PO$ satisfying $d(P,U) = \bott{\frac{x}{2}}$ and
$d(U,O) = \topp{\frac{x}{2}}$. 
Now there are two subcases to consider. 

\begin{figure}
\def\x{0.5}
\begin{center}
\begin{tikzpicture}[scale=0.75, decoration={ 
markings,
mark=
at position \x
with 
{ 
\arrow[scale=1.5]{stealth} 
} 
} 
]
\tikzstyle{vertex}=[circle,draw=black, fill=white, inner sep = 0.75mm]
\node (P)  [vertex,label={180:{\small $P$}}] at (0,0) {};
\node (Q)  [vertex,label={90:{\small $Q$}}] at (5,8.65) {};
\node (R)  [vertex,label={360:{\small $R$}}] at (10,0) {};
\node (M)  [vertex,label={180:{\small $M$}}] at (2.5,4.33) {};
\node (O)  [vertex,label={270:{\small $O$}}] at (4,0) {};
\node (U)  [vertex,label={270:{\small $U$}}] at (2,0) {};
\node (T_1)  [label={360:{\small $T_1$}}] at (1.7,1.5) {};
\node (T_2)  [label={360:{\small $T_2$}}] at (3.1,4.1) {};
\node (T_3)  [label={360:{\small $T_3$}}] at (6,2.5) {};
\draw [postaction={decorate}] (P)--(M) node[pos=0.5,left] {$\bott{\frac{k}{2}}$ \; };
\draw [postaction={decorate}] (M)--(Q) node[pos=0.5,left] {$\topp{\frac{k}{2}}$ \; };
\draw [postaction={decorate}] (O)--(M) node[pos=0.5,left] {$\leq \delta$};
\draw [postaction={decorate}] (Q)--(R) node[pos=0.5,right] {\; $l$};
\draw [postaction={decorate}] (O)--(R) node[pos=0.5,below] {$y$};
\draw [postaction={decorate}] (P)--(U) node[pos=0.5,below] {$\bott{\frac{x}{2}}$};
\draw [postaction={decorate}] (U)--(O) node[pos=0.5,below] {$\topp{\frac{x}{2}}$};
\draw [postaction={decorate}] (O)--(Q) node[pos=0.5,right] {$\leq \topp{\frac{k}{2}} + \delta$};
\end{tikzpicture}
\end{center}
\caption{Proof of Lemma~\ref{lem_tessellating}, Case (b).}
\label{fig_caseb}
\end{figure}
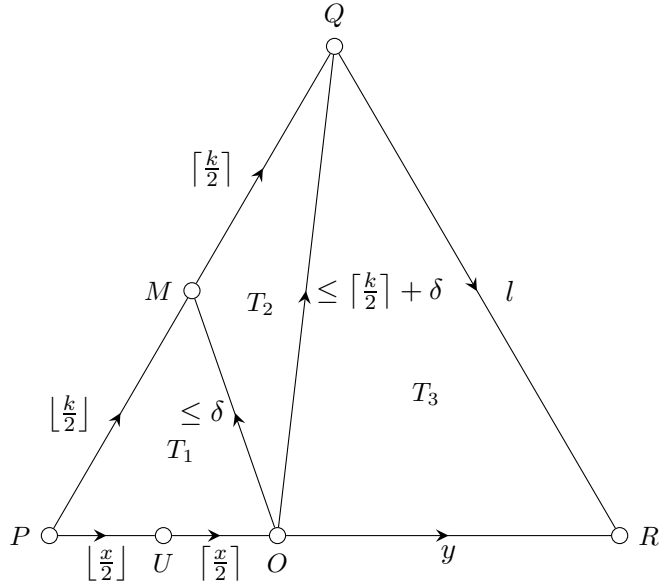

\

\noindent \textbf{\boldmath Case (b)(i): $\outball_\delta(U)$ intersects $OM$.} 
Let $S$ be a point in $\outball_\delta(U)$ which lies on $OM$. Choose geodesics $PU$, $UO$, $UM$, $US$,
$SM$ and $OS$ so that $PU \circ UO = PO$ and $OS \circ SM = OM$.
Now the directed geodesic triangle $POM$ is tessellated by the three triangles $Y_1 = PUM$, $Y_2=USM$ and
$Y_3=UOS$ (see Figure~\ref{fig_casebi}) and the sizes of these three triangles are bounded as follows: 
\begin{eqnarray*}
\Sigma(Y_1) & \leq & 
\bott{\frac{x}{2}} + 2\delta \leq \bott{\frac{k+l}{2}} + 2\delta \leq  \frac{1}{2} \Sigma(T) + 2\delta, \\
\Sigma(Y_2) & \leq & 2\delta, \\
\Sigma(Y_3) & \leq & \topp{\frac{x}{2}} + \delta \leq \topp{\frac{k+l}{2}} + \delta  \leq  \frac{1}{2} \Sigma(T) + \delta + 1.  
\end{eqnarray*}

\begin{figure}
\def\x{0.5}
\begin{center}
\begin{tikzpicture}[scale=0.75, decoration={ 
markings,
mark=
at position \x
with 
{ 
\arrow[scale=1.5]{stealth} 
} 
} 
]
\tikzstyle{vertex}=[circle,draw=black, fill=white, inner sep = 0.75mm]
\node (P)  [vertex,label={180:{\small $P$}}] at (0,0) {};
\node (M)  [vertex,label={90:{\small $M$}}] at (5,8.65) {};
\node (O)  [vertex,label={360:{\small $O$}}] at (10,0) {};
\node (U)  [vertex,label={270:{\small $U$}}] at (5,0) {};
\node (S)  [vertex,label={360:{\small $S$}}] at (8.75,2.165) {};
\node (Y_1)  [label={360:{\small $Y_1$}}] at (3,3.2) {};
\node (Y_2)  [label={360:{\small $Y_2$}}] at (6,2.8) {};
\node (Y_3)  [label={360:{\small $Y_3$}}] at (7.6,0.8) {};
\draw [postaction={decorate}] (P)--(M) node[pos=0.5,left] {$\bott{\frac{k}{2}}$ \; };
\draw [postaction={decorate}] (P)--(U) node[pos=0.5,below] {$\bott{\frac{x}{2}}$ \; };
\draw [postaction={decorate}] (U)--(O) node[pos=0.5,below] {$\topp{\frac{x}{2}}$ \; };
\draw [postaction={decorate}] (U)--(M) node[pos=0.5,left] {$\leq 2 \delta$};
\draw [postaction={decorate}] (U)--(S) node[pos=0.5,left] {$\leq \delta$ \; };
\draw [postaction={decorate}] (O)--(S) node[pos=0.5,right] {\; $\leq \delta$};
\draw [postaction={decorate}] (S)--(M) node[pos=0.5,right] {\; $\leq \delta$};
\end{tikzpicture}
\end{center}
\caption{Proof of Lemma~\ref{lem_tessellating}, Case (b)(i).}
\label{fig_casebi}
\end{figure}
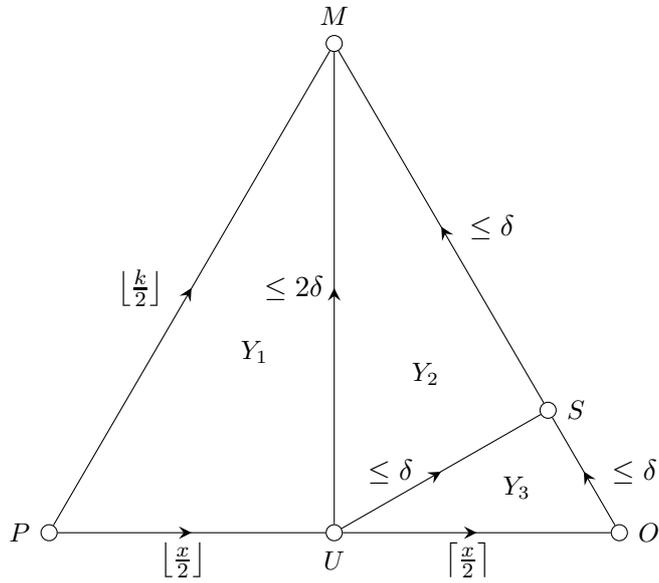

\

\noindent \textbf{\boldmath Case (b)(ii): $\inball_\delta(U)$ intersects $PM$.}
In this case choose $S$ in $\inball_\delta(U)$ which lies on $PM$. Choose geodesics $PU$, $UO$, $SU$, $SO$,
$PS$ and $SM$ so that $PU \circ UO = PO$ and $PS \circ SM = PM$.
Now the directed geodesic triangle $POM$ is tessellated by the three triangles $Y_1 = PSU$, $Y_2=SUO$ and $Y_3=SOM$ (see Figure~\ref{fig_casebii}) and the sizes of these three triangles are bounded as follows: 
\begin{eqnarray*}
\Sigma(Y_1) & \leq & \bott{\frac{k}{2}} + \delta \leq \frac{1}{2} \Sigma(T) + \delta, \\
\Sigma(Y_2) & \leq & \delta + \topp{\frac{x}{2}} \leq \delta + \topp{\frac{k+l}{2}} 
 \leq  \frac{1}{2}\Sigma(T) + \delta + 1, \\
\Sigma(Y_3) & \leq & \topp{\frac{x}{2}} + 2\delta \leq \topp{\frac{k+l}{2}} + 2\delta 
 \leq  \frac{1}{2}\Sigma(T) + 2\delta + 1.
\end{eqnarray*}
Thus, in both case (b)(i) and case (b)(ii), our triangle is tessellated by the triangles
$T_2$, $T_3$, $Y_1$, $Y_2$ and $Y_3$ which satisfy the required size bound.
\end{proof}

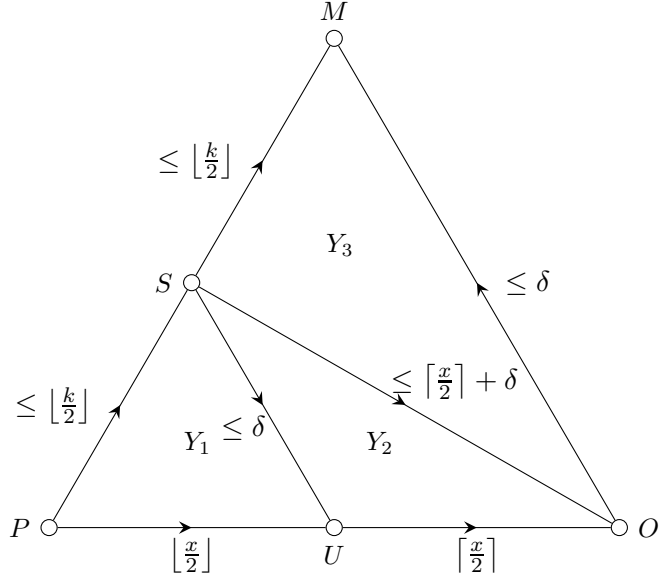
\begin{figure}
\def\x{0.5}
\begin{center}
\begin{tikzpicture}[scale=0.75, decoration={ 
markings,
mark=
at position \x
with 
{ 
\arrow[scale=1.5]{stealth} 
} 
} 
]
\tikzstyle{vertex}=[circle,draw=black, fill=white, inner sep = 0.75mm]
\node (P)  [vertex,label={180:{\small $P$}}] at (0,0) {};
\node (M)  [vertex,label={90:{\small $M$}}] at (5,8.65) {};
\node (O)  [vertex,label={360:{\small $O$}}] at (10,0) {};
\node (S)  [vertex,label={180:{\small $S$}}] at (2.5,4.33) {};
\node (U)  [vertex,label={270:{\small $U$}}] at (5,0) {};
\node (Y_1)  [label={360:{\small $Y_1$}}] at (2,1.5) {};
\node (Y_2)  [label={360:{\small $Y_2$}}] at (5.2,1.5) {};
\node (Y_3)  [label={360:{\small $Y_3$}}] at (4.5,5) {};
\draw [postaction={decorate}] (P)--(S) node[pos=0.5,left] {$\leq \bott{\frac{k}{2}}$ \; };
\draw [postaction={decorate}] (S)--(M) node[pos=0.5,left] {$\leq \bott{\frac{k}{2}}$ \; };
\draw [postaction={decorate}] (S)--(U) node[pos=0.6,left] {$\leq \delta$ };
\draw [postaction={decorate}] (S)--(O) node[pos=0.4,right] {$\; \; \leq \topp{\frac{x}{2}} + \delta$};
\draw [postaction={decorate}] (O)--(M) node[pos=0.5,right] {\; $\leq \delta$};
\draw [postaction={decorate}] (P)--(U) node[pos=0.5,below] {$\bott{\frac{x}{2}}$};
\draw [postaction={decorate}] (U)--(O) node[pos=0.5,below] {$\topp{\frac{x}{2}}$};
\end{tikzpicture}
\end{center}
\caption{Proof of Lemma~\ref{lem_tessellating}, Case (b)(ii).}
\label{fig_casebii}
\end{figure}

\begin{theorem}\label{thm_polytessellate}
Let $\Gamma$ be a strongly $\delta$-hyperbolic directed graph and $C > 8 \delta + 4$
a constant. Then every directed geodesic triangle $T$ in $\Gamma$ can be tessellated
by
$$5 \left(\frac{\Sigma(T)}{C - 8 \delta - 4} \right)^{\log_\frac{4}{3} 5}$$
or fewer geodesic triangles of size $C$ or less.
\end{theorem}
\begin{proof}
Let $T$ be a directed geodesic triangle in $\Gamma$. Then by
Lemma~\ref{lem_tessellating}, $T$ is tessellated by $T_1, \ldots, T_5$
where for all $i$
\[
\Sigma(T_i) \leq \frac{3}{4} \Sigma(T) + (2\delta + 1). 
\]
We iterate this procedure, at each stage subdividing every triangle from
the previous stage into five triangles in this way. Define a sequence $t_0$
of natural numbers by $t_0 = \Sigma(T)$ and
\[
t_{i+1} = \frac{3}{4} t_i + (2\delta + 1).
\]
for $i \leq n$.
A simple induction argument applying Lemma~\ref{lem_tessellating} shows that
$t_i$ is an upper bound on the size of the triangles obtained in the $i$th
iteration.

Now for each $k \in \mathbb{N}$ we have
\[
t_k \ = \ \left( \frac{3}{4} \right)^k t_0 + \left( \sum_{i=0}^{k-1} \left( \frac{3}{4} \right)^i  \right)
(2\delta + 1)
\ \leq \ 
\left( \frac{3}{4} \right)^k t_0 + 4 (2\delta + 1). 
\]
Let $D = C - 8 \delta - 4$ and let $N$ be the integer part of $\log_{4/3} \frac{t_0}{D} + 1$.
Then rearranging we have
$$N \geq \log_\frac{4}{3} \frac{t_0}{D}, \textrm{ so } 
\left( \frac{4}{3} \right)^N \geq \frac{t_0}{D}, \textrm{ so }
\left( \frac{3}{4} \right)^N t_0 \leq D = C - 8 \delta - 4$$
$$\textrm{ and hence } t_N \leq \left( \frac{3}{4} \right)^N t_0 + 8 \delta + 4 \leq C.$$
Thus, after $N$ iterations, we have tessellated $T$ with triangles of size
$C$ or less. Moreover, since at each stage we subdivide each triangle
into at most five triangles, the number of triangles in this tessellation is bounded
above by $5^N$, where
\begin{align*}
5^N &\leq 5^{ \left( \log_\frac{4}{3} \frac{t_0}{D} \right) + 1} 
= 5 \times 5^{\log_\frac{4}{3} \frac{t_0}{D}}
= 5 \times \left( \frac{t_0}{D} \right)^{\log_\frac{4}{3} 5} \\
&\leq 5 \left( \frac{\Sigma(T)}{C - 8 \delta - 4} \right) ^{\log_\frac{4}{3} 5}
\end{align*}
as required.
\end{proof}

Our main aim with Theorem~\ref{thm_polytessellate} was to give a reasonably
concise argument for the existence of a polynomial bound on the number of
triangles of fixed size required to tessellate a geodesic triangle, rather
than to optimise the degree of the polynomial.
The figure of $\log_{4/3} 5$ (which is approximately 5.6)
can probably be lowered significantly at the expense of lengthening the
proof, either by analysing more precisely the properties of the subdivision
given by Lemma~\ref{lem_tessellating}, or by considering alternative subdivisions.

Given two functions $f, g : \mathbb{N}\to \mathbb{N}$ we write $f \prec g$ if there exists
a constant $a$ such that $f(j) \leq ag(aj) + aj$ for all $j$. The functions
$f$ and $g$ are said to be of the same \textit{type}, written $f \sim g$, if $f \prec g$ and
$g \prec f$.

Now fix a monoid presentation $\langle A \mid R \rangle$.
If $u$ and $v$ are equivalent words then the \textit{area} $A(u,v)$ is
the smallest number of applications of relations from $R$ necessary to
transform $u$ into $v$.
The \textit{Dehn function} of a presentation $\langle A \mid R \rangle$ is
the function
$\delta : \mathbb{N} \to \mathbb{N}$ given by
$$\delta(n) = \max \lbrace A(u,v) \mid u,v \in A^*, u \equiv_R v, |u|+|v| \leq n \rbrace.$$
The Dehn function is a measure of the complexity of transformations between
equivalent words. 
The Dehn function depends on the presentation, but if $\delta$ and $\gamma$
are Dehn functions of different finite presentations for the same monoid
then $\delta \sim \gamma$ (see \cite{MadlenerOtto85, Pride95}).

A corollary of the above result is the following.

\begin{theorem}
\label{thm_DehnFnBound}
Let $M$ be a strongly $\delta$-hyperbolic left cancellative monoid. Then $M$ is finitely
presented with Dehn function bounded above by a polynomial of degree
$\log_\frac{4}{3} 5 + 1$.
\end{theorem}
\begin{proof}
Let $\Gamma$ be the Cayley graph of $M$, with respect to a generating set which
makes $M$ strongly $\delta$-hyperbolic. Choose an integer
$C > 8 \delta + 4$.

Suppose $(p, q)$ is a pair of parallel paths in $\Gamma$. Then by
Lemma~\ref{lem_FillingFish}, $(p,q)$ can be tessellated by at most
$|p|+|q|+1$ geodesic triangles of size at most $2(|p|+|q|)$. By
Theorem~\ref{thm_polytessellate}, each of these may be tessellated
by at most
$$5 \left(\frac{\Sigma(T)}{C - 8 \delta - 4} \right)^{\log_\frac{4}{3} 5}$$
geodesic triangles of size at most $C$. Thus, $(p,q)$ can be tessellated by at
$$5 (|p|+|q|+1) \left(\frac{2(|p|+|q|)}{C - 8 \delta - 4} \right)^{\log_\frac{4}{3} 5}$$
directed geodesic triangles of size at most $C$.

Each such triangle will correspond to a face in $K_C(\Gamma)$, so this shows
that $K_C(\Gamma)$ is simply connected with Dehn function bounded above by
a polynomial of degree $\log_\frac{4}{3} 5 + 1$.
It follows by the results of \cite{K_qsifp}
 that $M$ is finitely presented with Dehn function
bounded above by a polynomial of this degree.
\end{proof}

\begin{theorem}
\label{thm_polytimebound}
Let $M$ be a finitely generated, left cancellative, strongly $\delta$-hyperbolic monoid.
Then the word problem for $M$ lies in $\mathcal{NP}$.
\end{theorem}
\begin{proof}
By Theorem~\ref{thm_DehnFnBound} $M$ is finitely presented with polynomial
Dehn function. Let $\langle A \mid R \rangle$ be a finite presentation, and
$p : \mathbb{N} \to \mathbb{N}$ a polynomial upper bound on the corresponding Dehn function. Now given
words $u, v \in A^*$, one may check non-deterministically if $u = v$ in $M$
by guessing a sequence of relation applications of length $p(|u| + |v|)$ which can
be applied to $u$, and seeing if the result of applying them is $v$.
\end{proof}

\section{Deciding Green's Relations}\label{sec_greens}
The statements of our main results for monoids so far have been direct
analogues of known results in the group case, although the proofs have
been rather more involved. If geometric techniques are to have more than
a very limited application in semigroup theory, it is important that they
give insight into aspects of the structure theory of
semigroups which do not arise in groups such as, for example, the ideal
structure of a semigroup.
Recall that \textit{Green's relations} are a collection of equivalence
relations and pre-orders (reflexive, transitive binary relations) defined
on any monoid (or semigroup) which encapsulate the structure of its principle
left, right and two-sided ideals and maximal subgroups. They are a key
tool in modern semigroup theory, playing a pivotal role in almost every
area of the subject.

If $S$ is a monoid then we define pre-orders $\leq_{\gr}$, $\leq_{\gl}$,
$\leq_{\gj}$ and equivalence relations $\gr$, $\gl$, $\gj$, $\gh$ and $\gd$
by
\begin{itemize}
\item $a \leq_{\gr} b$ if and only if $aS \subseteq bS$;
\item $a \leq_{\gl} b$ if and only if $aS \subseteq bS$;
\item $a \leq_{\gj} b$ if and only if $SaS \subseteq SbS$;
\item $a \gr b$ if and only if $aS = bS$ (that is, if $a \leq_\gr b$ and $b \leq_\gr a)$;
\item $a \gl b$ if and only if $Sa = Sb$ (that is, if $a \leq_\gl b$ and $b \leq_\gl a)$;
\item $a \gj b$ if and only if $SaS = SbS$ (that is, if $a \leq_\gj b$ and $b \leq_\gj a)$;
\item $a \gh b$ if and only if $a \gr b$ and $a \gl b$; and
\item $a \gd b$ if and only if there exists $c \in S$ with $a \gl c$ and $c \gr b$.
\end{itemize}
A monoid is called $\mathscr{J}$-trivial if the $\gj$ relation (and hence
also the $\gd$, $\gl$, $\gr$ and $\gh$ relations) are the identity relation.

\subsection*{\boldmath Green's relations $\gr$ and $\gl$}

We shall now see how the triangle quasi-inequality for locally finite
strongly $\delta$-hyperbolic directed graphs, and more generally the
polygon quasi-inequality for strongly $\delta$-hyperbolic locally finite
directed graphs, can be usefully applied to prove decidability results
for Green's relations in strongly $\delta$-hyperbolic monoids. 

The questions of decidability, and the complexity of deciding, Green's relations have been considered
for semigroups defined by finite complete rewriting systems \cite{Otto84},
automatic monoids \cite{Otto07},  
word-hyperbolic semigroups \cite{CainArXiV}, and 
for the Thompson-Higman monoids 
\cite{Birget2010, Birget2011}.
In particular,  
it was shown in \cite{Otto84} that there exists monoids that are presented
by finite, length-reducing, and confluent string-rewriting systems (and
therefore in particular have solvable word problem) but have Green's
relations $\gr$ and $\gl$ that are undecidable. 
Also, in \cite{Otto07} examples are given of finitely generated monoids $M$ with word problem solvable in quadratic time but such that $\gr$ (respectively $\gl$) is undecidable. 
For strongly $\delta$-hyperbolic
monoids, the quasi-triangle inequality prevents this from happening.

Recall that a monoid $M$ generated by a finite subset $A$ has \textit{(right) indegree
bounded by a natural number $\alpha$} if one cannot choose generator $a \in A$ and
$\alpha + 1$ distinct elements $x_0, \dots, x_\alpha \in M$ such that
$x_0 a = x_1 a = \dots = x_\alpha a$. It is easily seen that the property of having
bounded indegree is independent of the choice of finite generating set, although the actual
bound may vary. Having bounded indegree is also equivalent to saying that the right
Cayley graph of $M$ (with respect to any or every choice of finite generating set) has
bounded valency. Notice that a right cancellative finitely generated monoid
always has bounded indegree, and indeed bounded indegree is often viewed
as a weak right cancellativity condition.

\begin{theorem}
\label{thm_GreensRRelation}
Let $M$ be a finitely generated strongly $\delta$-hyperbolic monoid with bounded
indegree. Then the problems of deciding the $\gr$-order $\leq_{\gr}$ and
$\gl$-order $\leq_{\gl}$ are reducible in non-deterministic linear time to
the word problem for $M$.
\end{theorem}
\begin{proof}
Let $A$ be a finite generating set with respect to which $M$ is strongly
$\delta$-hyperbolic.
Let $w, u \in A^*$, and suppose $\alpha \in A^*$ is of minimal length such that
$w \alpha = u$ in $S$.
Let $w'$ and $u'$ be a geodesic words representing the same elements
as $w$ and $u$ respectively. Then $(w',\alpha,u')$ labels a geodesic triangle
in the Cayley graph of $M$ with respect to $A$, so by
Theorem~\ref{thm_poly}, there is constant $K$, depending only
on $\delta$, the maximum indegree of $\Gamma$ and $|A|$, such that
$$|\alpha| \leq K(|w'| + |u'|) \leq K (|w| + |u|).$$

Thus, given $w,u \in A^*$, to test non-deterministically if $u \leq_\gr w$,
it suffices to guess a word $\alpha \in A^*$ of length at most $K(|w|+|u|)$,
and then test if $w \alpha = u$.

The proof for $\leq_\gl$ is entirely similar.
\end{proof}

Note that in Theorem~\ref{thm_GreensRRelation} we do not require the monoid
to be left cancellative, although we still have a weak \textit{right}
cancellativity assumption in the form of the bounded indegree hypothesis.
Combining Theorem~\ref{thm_GreensRRelation} with Theorem~\ref{thm_DehnFnBound}
we obtain the following.

\begin{theorem}
\label{thm_RisDecidable}
Let $M$ be a finitely generated left cancellative monoid of bounded indegree
which is strongly hyperbolic. Then the $\mathcal{L}$-order and
$\mathcal{R}$-order for $M$ are both in $\mathcal{NP}$.
\end{theorem}

Neither the left cancellativity nor the strong hyperbolicity condition in
Theorem~\ref{thm_RisDecidable} can be dropped. In
Section~\ref{sec_noncanc} below we shall see examples of finitely generated,
strongly $0$-hyperbolic which have unsolvable word problems and all of Green's
equivalence relations trivial, and hence also unsolvable.
Also, as mentioned above, it is well known that there exist finitely presented groups with unsolvable word problem. 
Let $G$ be such a group given by a finite monoid presentation  $\langle A | R \rangle$, and define $M = \langle A, h | R \rangle$ where $h$ is a symbol not in $A$. Then $M$ is a two-sided cancellative monoid (since it is the monoid free product of $G$ and the free monoid of rank one, both of which are cancellative). Moreover, for all words $u,w \in A^*$ we have
\[
hu \gl hw \Leftrightarrow w = u \ \mbox{in $G$, which is undecidable}, 
\]  
\[
uh \gr wh \Leftrightarrow w = u \ \mbox{in $G$, which is undecidable},
\]  
and
\[
huh \gj hwh \Leftrightarrow huh \gd hwh \Leftrightarrow w = u \ \mbox{in $G$, which is undecidable}.
\]  
Thus, none of the relations $\gr$, $\gl$, $\gj$ or $\gd$ is decidable in $M$.

\subsection*{\boldmath Green's relations $\gj$ and $\gd$}

Next we look at the relations $\gj$ and $\gd$. The following technical
lemma will be used to study Green's $\gj$-relation. Intuitively speaking,
it says that, in a geodesic quadrangle $(p,q,r,s)$, if the side $r$ is
sufficiently long then there will be a short path from $p$ to $r$. This
will be using for carving up geodesic quadrangles into smaller geodesic
quadrangles.

\begin{lemma}
\label{lem_technicallem}
Let $\Gamma$ be a strongly $\delta$-hyperbolic, locally finite directed graph with indegree and outdegree
bounded by $\alpha$, and let $(p,q,r,s)$ be a geodesic quadrangle. Then there are
polynomial-time computable
constants
$C_{\delta, |s|}$, depending on $\delta$ and $|s|$, and
$D_{\alpha, \delta, |q|, |s|}$, depending on $\alpha$, $\delta$, $|q|$ and
$|s|$, such that if $|r| > D_{\alpha, \delta, |q|, |s|}$ then there is a
geodesic path $t$ in $\Gamma$ satisfying 
$\iota t \in p$, 
$\tau t \in r$, 
$d(\tau t, \tau s) = C_{\delta, |s|}$,
and $|t| \leq 2\delta$.

Moreover, if the graph $\Gamma$ is fixed then $C_{\delta, |s|}$ may be
chosen to increase monotonically with $|s|$, and the $D_{\alpha, \delta, |q|, |s|}$
may be chosen to increase monotonically with $|q|$ and $|s|$, and to be bounded above
by a linear function of $|q| + |s|$.
\end{lemma}
\begin{proof}
Let $K_{\alpha, \delta}$ be the constant, given by Theorem~\ref{thm_poly}, depending on $\alpha$ and $\delta$,
and having the property that for any geodesic quadrangle in $\Gamma$ each side has length bounded by
$K_{\alpha, \delta}$ times the sum of the length of the other three sides. Define
\begin{equation}
\label{eqn_C}
C_{\delta, |s|}
 = 
\delta + |s| + 1,
\end{equation}
and
\begin{equation}
\label{eqn_D}
D_{\alpha,\delta,|q|,|s|}  
=  
K_{\alpha, \delta} (K_{\alpha,\delta} (|s|+\delta + C_{\delta,|s|} ) + \delta + |q| + |s|). 
\end{equation}
Since $K_{\alpha, \delta}$ is polynomial-time computable, it is clear that
these values are also computable in polynomial time.
Moreover, if the graph (and hence $\alpha$ and $\delta$) remain fixed then
clearly the values increase monotonically with the remaining variables, and
$D_{\alpha,\delta,|q|,|s|}$ can be bounded above by a linear function of $|q| + |s|$.

Let $x \in r$ with $d(x,\tau s) = C_{\delta, |s|}$. 
Let $u$ be a geodesic path from $\iota s$ to $\tau q$, and consider the
directed geodesic triangle $(u,r,s)$. 
Since $\Gamma$ is strongly $\delta$-hyperbolic either $\outball_\delta(x) \cap s \neq \varnothing$ or
else $\inball_\delta(x) \cap u \neq \varnothing$. However, the first of these possibilities cannot arise since 
\[
d(x, \tau s) = C_{\delta,|s|} > \delta + |s|, 
\]
and thus we conclude that $\inball_\delta(x) \cap u \neq \varnothing$, so we may choose $y \in u$ with
$d(y,x) \leq \delta$. Now consider the directed geodesic triangle $(p,q,u)$ and the point $y \in u$. Since $\Gamma$ is strongly $\delta$-hyperbolic, either (i) $\outball_\delta(y) \cap q \neq \varnothing$ or (ii) $\inball_\delta(y) \cap p \neq \varnothing$. 

Suppose, seeking a contradiction, that (i) holds. Then $d(y,\tau q) \leq \delta + |q|$. Now consider a directed geodesic quadrangle formed
by geodesics from $\iota s$ to $y$, $y$ to $x$, $x$ to $\tau s$ and $\iota s$ to $\tau s$.
By Theorem~\ref{thm_poly} we have
\[
d(\iota s, y) 
\leq
K_{\alpha, \delta} (|s| + \delta + C_{\delta,|s|}),
\]
and therefore
\[
|u| =
d(\iota s, \tau q) = d(\iota s, y) + d(y, \tau q) 
\leq
K_{\alpha, \delta} (|s| + \delta + C_{\delta,|s|})
+ \delta + |q|. 
\]
Then, again applying the triangle quasi-inequality from Theorem~\ref{thm_poly}, we know that
\[
|r| 
\leq K_{\alpha,\delta}(|u|+|s|) 
\leq K_{\alpha,\delta}(K_{\alpha,\delta}(|s| + \delta + C_{\delta,|s|}) + \delta + |q| + |s|),
\]
which is a contradiction, since $r$ was assumed to satisfy
\[
r > D_{\alpha,\delta,|q|,|s|}. 
\]
We deduce that (ii) $\inball_\delta(y) \cap p \neq \varnothing$, say $z \in \inball_\delta(y) \cap p$. 
Now the lemma follows simply by setting $t$ to be a geodesic path from $z$ to $x$. 
\end{proof}

The main application of the above lemma is to the proof of the following one, which will
be key to establishing the main results concerning $\gd$ and $\gj$.

\begin{lemma}
\label{lemma_DrelnThm}
Let $M$ be a left cancellative monoid which is strongly $\delta$-hyperbolic
with respect to a finite generating set $A$, and has indegree with respect to
$A$ bounded by $\alpha$.

Then for every pair $u,v \in A^*$ of geodesic words, there is a constant $F_{|A|,\alpha,\delta,|u|,|v|}$
(depending on, and polynomial-time computable from, $|A|$, $\alpha$, $\delta$, $|u|$, $|v|$)
such that if $u \gj v$ in $M$
then there exist words 
$a, b \in A^*$ 
such that 
$a u b = v$ and
$|a|,|b| \leq F_{|A|,\alpha,\delta,|u|,|v|}$.

If, in addition, $M$ is cancellative and $puq = v$ in $M$ for some units
$p$ and $q$, then $a$ and $b$ may be chosen to represent units.

Moveover, for a fixed monoid $M$, $F_{|A|,\alpha,|u|,|v|}$ is monotonically increasing as a function
of $|u|$ and $|v|$, and can be bounded above by a linear function of $|u| + |v|$.
\end{lemma}
\begin{proof} 
Let $a, b \in A^*$ be such that (i) $a u b = v$, (ii) if $M$ is cancellative
and it is possible to choose them so, $a$ and $b$ represent units, (iii)
$a$ is a geodesic word and (iv) $|b|$ is minimal subject to the preceding
three conditions (and in particular is a geodesic word).

Since $a$, $u$, $b$ and $v$ are all geodesic words, $aub = v$ in $M$ and
$M$ is left cancellative, they label the sides of a geodesic quadrangle in the
right Cayley graph of $M$ with respect to $A$; for brevity we identify the
words $a$, $u$, $b$
and $v$ with the geodesic paths they label in this quadrangle.

The intuition
of the
proof is that we work our way from $\tau v$ to $\tau u$ marking
points at regular intervals on the geodesic labelled by $b$. Each
of these points will be the terminal vertex of a path labelled by a word
$t_k$ from a point of the geodesic labelled by $a$ to that labelled
by $b$, and each of the words $t_k$ will have length bounded by $2\delta$.
If $b$ were excessively long then two of the words, $t_i$ and $t_j$ say,
must coincide. Then using left cancellativity we can perform a cut and paste
operation gluing $t_i$ along $t_j$ and in the process reduce the length of
the word $b$. See Figure~\ref{fig_Drel} for an illustration of the argument.   

In more detail, without loss of generality we may suppose that $\delta$ is
an integer. Let $W_{A,2\delta}$ be the
number of words over $A$ of length less than or equal to $2\delta$, so $W_{A,2\delta}$ is equal to $2\delta + 1$ if $|A|=1$, and is equal to $(|A|^{2\delta+1}-1)/(|A|-1)$ otherwise. 
First we define 
\[
E_{|A|,\alpha,\delta,|u|,|v|} = 
D_{\alpha,\delta,|u|,|v|} +
D_{\alpha,\delta,|u|,2\delta} +
C_{\delta,|v|} +
W_{A,2\delta}C_{\delta,2\delta}.
\]
Notice that if the monoid and generating set are fixed (so that $|A|$,
$\alpha$ and $\delta$ are constant) then $W_{A,2\delta}$ and
$C_{\delta,2\delta}$ are constant,
$D_{\alpha,\delta,|u|,|v|} + D_{\alpha,\delta,|u|,2\delta}$ is bounded above by a linear function in $|u| + |v|$ by Lemma~\ref{lem_technicallem} and since $\delta$ is constant, 
and  $C_{\delta,|v|}$ is bounded above by a linear function in $|u| + |v|$, again by Lemma~\ref{lem_technicallem} and since $\delta$ is constant. Thus, $E_{|A|,\alpha,\delta,|u|,|v|}$ can be bounded above by a linear
function of $|u| + |v|$.

We claim that
$$| b | \leq E_{|A|,\alpha,\delta,|u|,|v|}.$$
Indeed, suppose false for a contradiction.
Decompose $b = b_1 \circ c_1$ where $|c_1| = C_{\delta, |v|}$. Then by considering
the geodesic quadrangle $(a,u,b,v)$ and applying
Lemma~\ref{lem_technicallem}, there is a geodesic path labelled by
a word $t_1$ from a point
of $a$ to $\tau b_1$ with $|t_1| \leq 2 \delta$. Let $a_1$ be the suffix of $a$ leading from the 
start point of this path. 
Now starting from $i = 1$, we repeatedly take the geodesic quadrangle corresponding to
$(a_i, u, b_i, t_i)$, and write $b_i = b_{i+1} \circ c_{i+1}$
where $|c_{i+1}| = C_{\delta,|t_i|}$. So long as $|b_i| > D_{\alpha,\delta,|u|,|t_i|}$
we may use Lemma~\ref{lem_technicallem} to find a geodesic path, labelled by a word $t_{i+1}$ of
length at most $2 \delta$, from a point on $a_i$ to $\tau b_i$. Let $a_{i+1}$ be
the suffix of $a_i$ leading from the start point of this path, and then repeat for the next
value of $i$.

Notice that we can continue this process for at least $N = W_{A,2\delta} + 1$ steps, since 
\begin{align*}
|b| 								&> E_{|A|,\alpha,\delta,|u|,|v|} \geq D_{\alpha,\delta,|u|,|v|},
\end{align*}
which is needed for the first step, 
\begin{align*}
|b|   				&> E_{|A|,\alpha,\delta,|u|,|v|} \geq C_{\delta,|v|} + D_{\alpha,\delta,|u|,2\delta}
\geq C_{\delta,|v|} + D_{\alpha,\delta,|u|,|t_1|}, 
\end{align*}
which allows the second step, and 
\begin{align*}
|b|    		&> E_{|A|,\alpha,\delta,|u|,|v|} \geq 
C_{\delta,|v|} + W_{A,2\delta}C_{\delta,2\delta} + D_{\alpha,\delta,|u|,2\delta} \\
&\geq C_{\delta,|v|} + \sum_{i=1}^{k}C_{\delta,|t_i|} + D_{\alpha,\delta,|u|,|t_{k+1}|},
\end{align*}
for all $k = 1,2, \ldots, W_{A,2\delta}$, 
by the monotonically increasing nature of $C_{\delta,|s|}$ and $D_{\alpha,\delta,|u|,|v|}$, which 
allows us to carry out the the third up to the  $(W_{A,2\delta}+1)$th step. 

The result, after $N$ steps, is depicted in Figure~\ref{fig_Drel}.
At this point we have $N$ words $t_1$, $t_2$, \ldots, $t_N$ all of length less than $2\delta$.
By the choice of $W_{A,2\delta}$ and the pigeon-hole principle, it follows that there exists
an $i<j$ such that $t_i = t_j$. 

This gives rise to a decomposition $a = a' a'' a'''$,
where $a'$ labels the path from $\iota a$ to $\iota t_i$, $a''$
the path from $\iota t_i$ to $\iota t_j$, and $a'''$ the path from
$\iota t_j$ to $\tau a$. Similarly, we write $b = b' b'' b'''$,
where $b'$ labels the path from $\iota b$ to $\tau t_j$, $b''$ the
path from $\tau t_j$ to $\tau t_i$ and $b'''$ the path from $\tau t_i$
to $\tau b$. Note that since $i \neq j$, $|b''| > 0$. Now we have
$a' a'' a''' u b' = a' a'' t_j$ in the monoid,
and hence by left cancellativity, $a''' u b' = t_j$. But now
\[(a' a''') u (b' b''') = a' (a''' u b') b''' = a' t_j b''' = a' t_i b''' = v.\]

But $b''$ is non-empty, so $|b' b'''| < |b|$. Moreover, if $M$ is cancellative
and $b$ represents a unit then $b$, and hence $b'b'''$ will be entirely composed of generators
representing units, so $b'b'''$ will also be a unit. This contradicts the
minimality in the choice of $b$, and establishes the claim that
$| b | \leq E_{|A|,\alpha,\delta,|u|,|v|}$.

Now, if we let $\beta = \max(\alpha, |A|)$ then by
Theorem~\ref{thm_poly} (the polygon inequality) we have
\[
|a| \leq K_{\beta, \delta} (|u| + |v| + |b|).
\] 
so it suffices to set
$$F_{|A|,\alpha,\delta,|u|,|v|}
=
(K_{\beta,\delta}+1)(|u| + |v| + E_{|A|,\alpha,\delta,|u|,|v|} ).$$
The fact that $F_{|A|,\alpha,\delta,|u|,|v|}$ is polynomial-time
computable, monotonically increasing and bounded above by a linear
function follow from the corresponding properties for $E_{|A|,\alpha,\delta,|u|,|v|}$.
\end{proof}

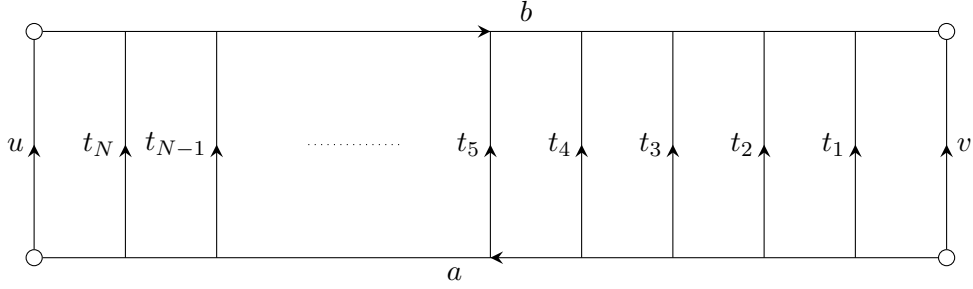
\begin{figure}
\def\x{0.5}
\begin{center}
\begin{tikzpicture}[scale=0.6, rotate=90, decoration={ 
markings,
mark=
at position \x
with 
{ 
\arrow[scale=1.5]{stealth} 
} 
} 
]
\tikzstyle{vertex}=[circle,draw=black, fill=white, inner sep = 0.75mm]
\node (P00)  [vertex,label={180:{}}] at (0,0) {};
\node (P50)  [vertex,label={180:{}}] at (5,0) {};
\node (P020)  [vertex,label={180:{}}] at (0,20) {};
\node (P520)  [vertex,label={180:{}}] at (5,20) {};
\draw [postaction={decorate}] (P00)--(P50) node[pos=0.5,right] {$v$};
\draw [postaction={decorate}] (P00)--(P020) node[pos=0.5,below] {$a$ \; \; \; \; };
\draw [postaction={decorate}] (P020)--(P520) node[pos=0.5,left] {$u$};
\draw [postaction={decorate}] (P520)--(P50) node[pos=0.5,above] { \; \; \; \; $b$};
\draw [postaction={decorate}] (0,2)--(5,2) node[pos=0.5,left] {$t_1$};
\draw [postaction={decorate}] (0,4)--(5,4) node[pos=0.5,left] {$t_2$};
\draw [postaction={decorate}] (0,6)--(5,6) node[pos=0.5,left] {$t_3$};
\draw [postaction={decorate}] (0,8)--(5,8) node[pos=0.5,left] {$t_4$};
\draw [postaction={decorate}] (0,10)--(5,10) node[pos=0.5,left] {$t_5$};
\draw [postaction={decorate}] (0,18)--(5,18) node[pos=0.5,left] {$t_{N}$};
\draw [postaction={decorate}] (0,16)--(5,16) node[pos=0.5,left] {$t_{N-1}$};
\draw [dotted] (2.5,12)--(2.5,14);
\end{tikzpicture}
\end{center}
\caption{
Proof of Lemma~\ref{lemma_DrelnThm}.
}
\label{fig_Drel}
\end{figure}

For our main result concerning the $\gd$ relation, we shall need the
following elementary fact about cancellative monoids, a proof of which
we include for completeness.

\begin{lemma}
\label{lem_CancMonoid}
Let $M$ be a cancellative monoid and let $a,b \in M$. Then $a \gd b$ if
and only if there are units $q,r \in U(M)$ satisfying $qar=b$. 
\end{lemma}
\begin{proof}
If $a \gd b$ then by definition there exists $c \in M$ such that $a \gl c$
and $c \gr b$. The former means there are $p,q \in M$ with $pc = a$
and $qa = c$, so $qpc = c = 1.c$ and $pqa = a = 1.a$, which by cancellativity
implies
$pq = qp = 1$, so $q$ is a unit. A dual argument gives $b = cr$
for some unit $r$, and now $b = cr = qar$ as required. The converse is
immediate.
\end{proof}

We are now ready to prove our main theorem about $\gd$ and $\gj$.

\begin{theorem}
\label{thm_JisDecidable}
Let $M$ be a finitely generated, left cancellative monoid of bounded
indegree which is
strongly hyperbolic. Then the $\mathcal{J}$-order is 
in $\mathcal{NP}$.

If, moreover, $M$ is cancellative, then the
$\mathcal{D}$-relation is in $\mathcal{NP}$.
\end{theorem} 
\begin{proof}
Let $A$ be a finite generating set with respect to which $M$ is
strongly $\delta$-hyperbolic, and $\alpha$ a corresponding bound on
the indegree. Suppose $u, v \in A^*$ are such that $v \leq_\gj u$ in $M$. Let $u'$ and
$v'$ be geodesic words representing the same elements as $u$ and $v$ respectively.
Then applying Lemma~\ref{lemma_DrelnThm} we have $a u' b = v'$ in $M$ for some
elements $a, b \in A^*$ with
$$|a|,|b| \ \leq \ F_{|A|,\alpha,\delta,|u'|,|v'|} \  \leq \ F_{|A|,\alpha,\delta,|u|,|v|},$$
where the last inequality follows from the monotonicity of
$F_{|A|,\alpha,\delta,|u|,|v|}$ as a function of $|u|+|v|$. Since $u = u'$
and $v=v'$ in $M$, it follows also that $aub = v$ in $M$.

Thus, given words $u$ and $v$, to check non-deterministically if $v \leq_\gj u$
it suffices compute the constant $F = F_{|A|,\alpha,\delta,|u|,|v|}$, guess
words $a$ and $b$ of length no more than $F$, and test if $aub = v$ in $M$.

By Lemma~\ref{lemma_DrelnThm}, the computation of $F$ can be performed in
polynomial time, and $F$ itself is bounded by a polynomial function of
$|u|+|v|$. Thus, the words $a$ and $b$ to be guessed have polynomial length.
Finally, the test of whether $aub = v$ in $M$ can be performed in
non-deterministic polynomial time by Theorem~\ref{thm_polytimebound}.
Thus, the whole procedure is possible in (non-deterministic) polynomial
time, so the $\mathcal{J}$-order is in $\mathcal{NP}$.

For the case of $\gd$ in a cancellative monoid, let $B$ be the set of
generators in $A$ which represent units in $M$. It is easily seen that
$B^*$ is exactly the set of words in $A^*$ representing units in $M$.
Thus, by Lemmas~\ref{lemma_DrelnThm} and \ref{lem_CancMonoid}, two words $u$
and $v$ represent $\gd$-related elements
if and only if there are words $a, b \in B^*$ with
$|a|,|b| \leq F$ and $a u b = v$ in $M$. So to check if $u \gd v$ in $M$,
we can use exactly the 
same procedure as above but considering only those $a$ and $b$ in $B^*$.
\end{proof}

We note that the second part of Theorem~\ref{thm_JisDecidable} is not
a trivial consequence of the first, since there are finitely presented cancellative monoids for which the relations $\mathcal{D}$ and $\mathcal{J}$ do not coincide. For instance consider the monoid $M$ defined by the presentation 
\[
\langle \; x,y,a,b \; | \; axb=y, \ ayb=x \; \rangle. 
\]
It follows from results of Adjan \cite{Adyan60}, since the relations have neither left cycles or right cycles, 
that $M$ is a cancellative monoid (in fact, it is group embeddable). 
It is a straightforward exercise to verify that in this monoid the elements represented by $x$ and $y$ are $\mathcal{J}$-related, but they are not $\mathcal{D}$-related (since the group of units is trivial, and the $\mathcal{R}$- and $\mathcal{L}$-relations are trivial).

\section{Monoids which are Not Left Cancellative}\label{sec_noncanc}

The definition of strong $\delta$-hyperbolicity makes sense for arbitrary
directed graphs, and hence for arbitrary monoids, but most of our results
so far have required a left cancellativity assumption on the monoid.
One might reasonably ask what can be deduced about a more general finitely
generated monoid, given only the information that it is strongly $\delta$-hyperbolic.
In this section we present a class of finitely
generated monoids which are right cancellative but not left cancellative and
which, although strongly $0$-hyperbolic, need not even be recursively presentable.
Indeed, the Cayley graph of each monoid is in most respects very like a
tree, being a rooted directed tree with the addition of some
duplicate edges, 
and so is likely to satisfy any reasonable geometric
definition of hyperbolicity. 
We believe this provides very strong evidence that geometric hyperbolicity
conditions cannot provide the same kind of information about general monoids
that they do about groups, and perhaps about other restricted classes of
monoids.

Given $w = a_1a_2\cdots a_r$ and $k \in \{1,\ldots,r\}$ write $w[k] = a_1 \cdots a_k$. 
For each subset $I$ of $\mathbb{N}$ define
\[
M_I = \langle \; a,b,c,d \; | \; ab^ic=ab^id \ (i \in I) \; \rangle. 
\]
\begin{lemma}
\label{lem_basic}
Let $w,u \in \lbrace a,b,c,d \rbrace^*$. If $w=u$ in $M_I$ then $|w|=|u|$ and for all $k \in \{1,\ldots,|w|\}$ we have 
$w[k] = u[k]$ in $M_I$. 
\end{lemma}
\begin{proof}
This can be shown by a straightforward induction on the number of applications
of relations required to transform $w$ into $u$. It suffices just to consider
the case where $w$ and $u$ are separated just by the application of a single
relation, so $w = \alpha ab^i c \beta$, $u = \alpha a b^i d \beta$.
But in this case the result clearly holds. 
\end{proof}

\begin{corollary}\label{cor_rightcanexample}
$M_I$ is right cancellative and $\gj$-trivial.
\end{corollary}
\begin{proof}
Suppose $x, y, z \in M_I$ are such that $xz = yz$. Choose words $u, v, w$
respectively to represent them, so that $uw = vw$ in $M_I$. Then by
Lemma~\ref{lem_basic}, $|uw| = |vw|$, so $|u| = |v|$. Now taking 
$k = |u| = |v|$ and using Lemma~\ref{lem_basic} again we have that
$u = v$ in $M_I$, so $x = y$.

If $u \gj v$ in $M_I$ then we have $puq = v$ and $rvs = u$ in $M_I$ for some words
$p$, $q$, $r$ and $s$. But now by Lemma~\ref{lem_basic} again,
$|u| = |rvs| = |rpuqs|$, which means $p$,$q$,$r$ and $s$ are all the
empty word and $u = v$ in $M_I$.
\end{proof}

\begin{proposition}
\label{prop_0hyperExample}
For every subset $I$ of $\mathbb{N}$, $M_I$ is a finitely generated strongly $0$-hyperbolic monoid.  
\end{proposition}
\begin{proof}
Let $\Gamma$ be the right Cayley graph of $M_I$ with respect to $A=\{a,b,c,d\}$, let $(p,q,r)$ be a geodesic directed triangle in $\Gamma$, and let $s$ be a geodesic path in $\Gamma$ from $1_M$ to $\iota p$. Suppose that $w_s$, $w_p$, $w_q$ and $w_r$ are the words labelling the paths $s$, $p$, $q$ and $r$, respectively. By Lemma~\ref{lem_basic} we have $w_s w_p w_q [k] = w_s w_r [k]$ for all $k \in \{1,\ldots, |w_s w_r|\}$. From this it follows that the set of vertices visited by the path $p \circ q$ is equal to the set of vertices visited by the path $r$. It is then immediate that the geodesic triangle $(p,q,r)$ is strongly $0$-hyperbolic. 
\end{proof}

\begin{corollary}
There exists a finitely generated, right cancellative, $\gj$-trivial
strongly $0$-hyperbolic monoid
$M$ which is not recursively presentable (and hence has word problem
and all of Green's relations undecidable).
\end{corollary}
\begin{proof}
Let $I$ be a subset of the natural numbers which is not recursively enumerable,
and set $M=M_I$. Then $M$ is finitely generated by definition, right
cancellative and $\gj$-trivial by Corollary~\ref{cor_rightcanexample} and strongly $0$-hyperbolic
by Proposition~\ref{prop_0hyperExample}. Suppose for a contradiction that
$M$ were recursively presentable. Then by enumerating a presentation and its
consequences, we could enumerate all relations which hold in $M$. In
particular we could enumerate those relations of the form $ab^ic = ab^id$
which hold in $M$. But it is easily seen that such a relation holds if and
only if $i \in I$, so this would allow us to enumerate $I$, contradicting
the assumption that $I$ is not recursively enumerable.

Since $M$ is not recursively presentable, it does not have solvable
word problem. And since $M$ is $\gj$-trivial, a decision process
for any of Green's equivalence or pre-order relations would permit
the solution of the word problem.
\end{proof}

\bibliographystyle{plain}
\bibliography{mark}

\def\cprime{$'$} \def\cprime{$'$}

\end{document}